\documentclass{amsart}
\usepackage{a4wide}
\usepackage[utf8]{inputenc}
\usepackage{amsthm,amsfonts,pifont,amsmath,graphicx}
\usepackage[colorlinks,linkcolor=red,anchorcolor=green,citecolor=blue,backref=page]{hyperref}
\usepackage{geometry}
\usepackage[foot]{amsaddr}

\newcommand{\R}{\mathbb{R}}

\newcommand{\X}{\mathcal{X}}

\newtheorem{teo}{Theorem} 
\newtheorem{prop}{Proposition} 
\newtheorem{proposition}{Proposition}
\newtheorem{lem}{Lemma} 

\newtheorem{rem}{Remark}

\title{Controllability of a Semilinear System of Parabolic Equations with Nonlocal Terms}

\author{Juan Límaco}
\address[A1]{Universidade Federal Fluminense, CEP 24210-201, Niterói- RJ, Brasil}
\email[A1]{jlimaco@id.uff.br}

\author{Rafael Martins Lobosco}
\address[A2]{Instituto Federal de Educação, Ciência e Tecnologia do Rio de Janeiro, CEP 24425-000, São Gonçalo-RJ, Brasil}
\email[A2]{rafael.lobosco@ifrj.edu.br}

\author{Luis P. Yapu}
\address[A3]{Universidade Federal Fluminense, CEP 24210-201, Niterói- RJ, Brasil and
Friedrich-Alexander Universität Erlangen-Nürnberg (FAU), Chair of Dynamics, Control, Machine Learning and Numerics, Erlangen, Germany
}
\begin{document}
%\maketitle


\begin{abstract}
This paper extends our previous controllability results for a class of coupled linear parabolic systems with nonlocal interactions, motivated by applications in finance such as generalized Black--Scholes models. We establish local null controllability at a fixed time $T>0$ for a class of semilinear, nonlocally coupled systems driven by a single internal control acting on one component. The proof combines Kakutani's fixed-point theorem with a controllability/observability estimate for the associated linearized dynamics. In addition, we obtain controllability for a broader class of linear systems than those considered in the first article. The paper concludes with remarks on boundary controllability within the same nonlocal framework and with perspectives for future research.
\end{abstract}

\maketitle

\section{Introduction}
\label{sec:intro}

For any $T>0$, let us consider the spacial domain $\Omega = (0,1)$ and $Q = (0,1) \times (0,T)$. We study the null controllability of the following systems of semilinearly coupled equations with kernel terms,
\begin{equation} \label{sistema0}
\begin{cases}
y_t=a_1  y_{xx}+b_1 y_x + c_1 y +\lambda_1(t)\int_0^1 J_1(\zeta-x)y(\zeta,t) d\zeta + F(y,z) + \nu 1_\omega, \ \ \mbox{ in } Q,\\
z_t=a_2  z_{xx}+b_2 z_x + c_2 z +\lambda_2(t)\int_0^1 J_2(\zeta-x)z(\zeta,t)d\zeta+ G(y,z), \ \ \mbox{ in } Q,\\
y(0,t) = y(1,t) = z(0,t) = z(1,t)  = 0 ,\ \ 
\mbox{ for } 0<t<T,\\
y(x,0) = y_0(x), \  z(x,0) = z_0(x), \mbox{ in }  \Omega.
\end{cases}
\end{equation}
where $\nu$ is a control acting on the open set $\omega \times (0,T)$ with $ \omega \subset (0,1)$, $y$ and $z$ are state functions, and, for $i=1,2$, $\lambda_i(t)$ are time dependent functions. In \eqref{sistema0} $a_i$, $b_i$ and $c_i$ are real constants such that $a_i >0$, for $i = 1,2$. The precise hypothesis about the functions $F$ and $G$ will be stated in our main theorem.

The present system \eqref{sistema0} aims to analyze the null controllability within a broader class of equations that extends the model below considered in the previous article of the authors \cite{Limaco-Lobosco-Yapu_24}: 

\begin{equation} \label{sistema_anterior}
\begin{cases}
u_t=a_1  u_{xx}+b_1 u_x + c_1 u +\lambda_1\int_0^1 J_1(z-x)u(z,t) dz+\\
\qquad \qquad \qquad  \qquad \qquad \qquad \quad  \qquad \qquad \qquad  + q_{11}u+q_{12}v + \nu 1_\omega, \ \ \mbox{ in } Q,\\
v_t=a_2  v_{xx}+b_2v_x +c_2v +\lambda_2\int_0^1J_2(z-x)v(z,t)dz+q_{21}u+q_{22}v, \ \ \mbox{ in } Q,\\
u(0,t) = u(1,t)=v(0,t) = v(1,t)  = 0 ,\ \ 
\mbox{ for } 0<t<T,\\
u(x,0) = u_0(x), \  v(x,0) = v_0(x), \mbox{ in }  \Omega,
\end{cases}
\end{equation}
where $\nu$ is a control acting on the open set $\omega \times (0,T)$, $u$ and $v$ are state functions, $\lambda_i(t)$ are time dependent functions with $i=1,2$ and $q_{ij}$, $a_i$, $b_i$ and $c_i$, with $i,j = 1,2$, are real constants such that $a_i >0$;
$q_{ij}$ are transition rates; $q_{ij}\geq 0$ if $i \neq j$; and  $q_{i1}+q_{i2}=0$. 

Problems of the type \eqref{sistema_anterior} were studied, for instance, when the control $\nu$ is not considered and 
$$J_i(x) = \lambda_i C_i e^{-\frac{(x-d_i)^2}{2k_i^2}},$$
with $\lambda_i$, $C_i$, $d_i$ and $k_i$ are constants. In that case, we have the model proposed by \cite{acoplado} for pricing of European, American and Butterfly options whose asset price dynamics follow the regime switching jump diffusion process. This model generalized the classical work \cite{BS} where the Black-Scholes equation was established for derivative pricing of European options. Due to the complexity of financial markets, that generalization became necessary in order to have the ability of efficiently interpret the economic cycles and the changes in the financial time series data due to the regime shifts, as was analysed in \cite{acoplado}. For other financial options pricing models, see the articles \cite{BS2, BS7, BS14, FM1}. 

Concerning partial differential equations with nonlocal terms (kernels) there are controllability results for one partial differential equation. In fact, for the linear equation
\begin{equation*} 
%\label{sistema22}
\begin{cases}
u_t= u_{xx} + \int_0^1 K(x,z,t)u(z,t) dz + \nu 1_\omega, \ \ \mbox{ in } Q,\\
u(0,t) = u(1,t)= 0 ,\ \ \mbox{ for } 0<t<T,\\
u(x,0) = u_0(x), \mbox{ in }  \Omega,
\end{cases}
\end{equation*}
it is known that the equation is null controllable when the nonlocal term is independent of time and analytic \cite{fernandez2016null}, see also \cite{lissy2018internal} for a coupled system. For a kernel independent of time and with separable variables, the controllability was established in \cite{micu2018local}. For one equation with kernel depending on time,  Biccari--Hernandez-Santamaría \cite{biccari2019} showed the null controllability supposing an exponential decay at the ends of the interval $[0,T]$ such as the condition  \eqref{H_0} below.  
With the same decay condition, and the smallness condition \eqref{H2}, the authors showed in \cite{Limaco-Lobosco-Yapu_24} the controllability of a coupled linear system.
On the other hand, Nina-Huaman and the first author showed in \cite{Limaco-Dany-25} a Carleman inequality and null controllability of a parabolic PDE whose linearization about a trajectory has a non-local integral term, but with kernel $K$ identically 1. In this case no decay condition is needed.  

In the present article, we investigate whether the controllability properties established for the original system remain valid under the introduction of additional terms $F$ and $G$, thereby providing a natural generalization of the earlier results. In fact, the conditions imposed in the funcions $F$ and $G$ for the system \eqref{sistema0} permit to include the terms of system \eqref{sistema_anterior} that do not appear explicitly in system \eqref{sistema0} and to introduce nonlinearities. 

We consider the case where the functions $J_i\in L^{\infty}(-1,1)$ and $\lambda_i\in L^{\infty}(0,T)$ satisfy the following condition of exponential decay at the extremity points of the interval $[0,T]$:
\begin{equation} \label{H_0}
\overline{K}=\sup_{t \in  [0,T]} exp\left(\frac{2\sigma^-}{t(T-t)} \right)\lambda^2_i(t)\int_{-1}^1 \vert J_i(z) \vert^2 dz < + \infty 
\end{equation}
for $i=1,2$, where $\sigma^-$ is a positive constant that will be defined later.   

In order to get null controllability with only one control, we need the following additional hypothesis on the decay of the first kernel:
\begin{equation} \label{H2}
\sup_{(x,t) \in [0,1]\times [0,T]} exp\left(2s \alpha^{-}(t) \right)\lambda_1^2(t)\int_{0}^1 \vert J_1(x-z) \vert^2 dz < \overline{\delta},
\end{equation}
for $\overline{\delta}>0$ which will be chosen small enough and $\alpha^{-}$ will be defined later.

Our main result is:
\begin{teo}\label{maintheorem0}
Let $T>0$. Suppose that hypothesis \eqref{H_0} and \eqref{H2} are satisfied. Then there exists $\delta = \delta(T)>0$ such that, for any initial data $y_0, z_0 \in L^2(\Omega)$ satisfying 
\begin{equation}
\|y_0\|_{L^2(\Omega)} + \|z_0\|_{L^2(\Omega)} < \delta,
\end{equation}
and for any $F(r,s), G(r,s) : \R \times \R \to \R$, $C^1$-functions fulfilling
\begin{equation}\label{BoundM}
\max \left\{ \left|\frac{\partial F}{\partial r} (r,s)\right|, \left|\frac{\partial F}{\partial s} (r,s)\right|,  \left|\frac{\partial G}{\partial r} (r,s)\right|, \left|\frac{\partial G}{\partial s} (r,s)\right|
\right\} \leq M,
\end{equation}
for all $(r,s) \in \R^2$, with $\frac{\partial G}{\partial r} (r,s) \neq 0$ in \eqref{sistema0}, $F(0,0)=0$, $G(0,0)=0$, then there exists a control $\nu \in L^2(\omega \times (0,T))$ such that the solution of the system \eqref{sistema0} verifies $y(x,T)=0$ and $z(x,T)=0$.
\end{teo}

This paper is organized as follows. In Section~\ref{sec:linear_case0} we formulate the linear problem associated with system~\eqref{sistema0}. This system differs from that in~\cite{Limaco-Lobosco-Yapu_24} as its coefficients are variable rather than constant, and the arguments from the earlier article are adapted here in detail. In Section~\ref{sec:proof_nonlinear0} we prove the main result, Theorem~\ref{maintheorem0}, for our system~\eqref{sistema0}, via Kakutani’s fixed-point theorem. In Section~\ref{sec:add_coments} we present related problems and additional remarks.

Throughout the paper, we use \(C\) and \(C_i\) (\(i=1,2,\ldots\)) to denote generic positive constants whose value may change from line to line. When it is necessary to emphasize parameter dependence, we indicate it explicitly; for example, we write \(C=C(\omega,T)\) to mean that \(C\) depends only on \(\omega\) and \(T\).

\section{Analysis of the Controllability of the Linearized System of (1)}
\label{sec:linear_case0}

Considering the hypotheses of the Theorem \ref{maintheorem0}, we have $F(0,0)=0$ and $G(0,0)=0$, then
\begin{equation}
\label{eq:decomp_F_G}
F(r,s) = a(r,s)r + b(r,s)s, \qquad
G(r,s)=c(r,s)r + d(r,s)s,    
\end{equation}
where 
$$
a(r,s) = \int_0^1 \frac{\partial F}{\partial r}(\kappa r,\kappa s) d\kappa, \quad b(r,s) = \int_0^1 \frac{\partial F}{\partial s}(\kappa r,\kappa s) d\kappa,
$$
$$
c(r,s) = \int_0^1 \frac{\partial G}{\partial r}(\kappa r,\kappa s) d\kappa, \qquad d(r,s) = \int_0^1 \frac{\partial G}{\partial s}(\kappa r,\kappa s) d\kappa.
$$

Taking into account the hypothesis \eqref{BoundM}, we have
$$
|a(r,s)| \leq \int_0^1 \left|\frac{\partial F}{\partial r}(\kappa r, \kappa s)\right| d\kappa \leq \int_0^1 M d\kappa = M,
$$
analogously $|b(r,s)| \leq M$, $|c(r,s)| \leq M$ and $|d(r,s)| \leq M$.

We denote $\X = L^2(Q)\times L^2(Q)$. 
For $(\bar y, \bar z) \in \X$, let us consider the solutions $(y,z)$ of the auxiliary linear system
\begin{equation}\label{eq:linearized_fixed_point}
\begin{cases}
y_t=a_1  y_{xx}+b_1 y_x + c_1 y +\lambda_1\int_0^1 J_1(\zeta-x)y(\zeta,t) d\zeta + \tilde a(x,t)y + \tilde b(x,t)z + \nu 1_\omega, \ \ \mbox{ in } Q,\\
z_t=a_2  z_{xx}+b_2 z_x +c_2z +\lambda_2\int_0^1J_2(\zeta-x)z(\zeta,t)d\zeta+ \tilde c(x,t)y + \tilde d(x,t)z, \ \ \mbox{ in } Q,\\
y(0,t) = y(1,t)=z(0,t) = z(1,t)  = 0 ,\ \ 
\mbox{ for } 0<t<T,\\
y(x,0) = y_0(x), z(x,0) = z_0(x), \mbox{ in }  \Omega,
\end{cases}
\end{equation}
where the coefficients are defined by
\begin{equation*}
%\label{eq:bound_M}
\begin{split}
\tilde a(x,t) = a(\bar y(x,t),\bar z(x,t)), \qquad \tilde b(x,t) = b(\bar y(x,t),\bar z(x,t)), \\
\tilde c(x,t) = c(\bar y(x,t),\bar z(x,t)), \qquad \tilde d(x,t) = d(\bar y(x,t),\bar z(x,t)).
\end{split}
\end{equation*}

Thus,
\begin{equation*}
%\label{eq:bound_M}
\begin{split}
|\tilde a(x,t)| = |a(\bar y(x,t),\bar z(x,t))| \leq M, \qquad |\tilde b(x,t)| = |b(\bar y(x,t),\bar z(x,t))| \leq M, \\
|\tilde c(x,t)| = |c(\bar y(x,t),\bar z(x,t))| \leq M, \qquad |\tilde d(x,t)| = |d(\bar y(x,t),\bar z(x,t))| \leq M.
\end{split}
\end{equation*}

Under the decay hypothesis \eqref{H_0} and \eqref{H2}, supposing a non-vanishing condition on 
$c(r,s)$, we recall the following controllability result for the linear system:   

\begin{teo}\label{theorem_linear} 
If the hypothesis \eqref{H_0} and \eqref{H2} are satisfied, $\tilde c \in W^{2,\infty}(Q)$ with $\tilde c(x,t) \neq 0$ in $\bar \omega \times (0,T)$, then there exists a control $\nu \in L^2(\omega \times (0,T))$ such that the solution of the system \eqref{eq:linearized_fixed_point} verifies $y(x,T)=z(x,T)=0$ and
 for a constant $C=C(M,\omega,T)$ we have an estimate for the control of the form
\begin{equation}\label{ControlEstimate}
\|\nu\|_{L^2(\omega\times(0,T))} \leq C (\|y_0\|_{L^2(\Omega)} + \|z_0\|_{L^2(\Omega)}).    
\end{equation}
\end{teo}

Theorem \ref{theorem_linear} was already proved in \cite{Limaco-Lobosco-Yapu_24} for the case in which the coefficients of the system are constant. That proof omits the explicit control estimate \eqref{ControlEstimate}, which is indispensable for the fixed-point argument. In our setting, the hypotheses on $F$ and $G$ imply that some coefficients of the linearized problem \eqref{eq:linearized_fixed_point} are bounded functions rather than constants, but the argument is very similar except for one extra step (we must establish the uniform bound \eqref{ControlEstimate}) and some adaptations. We present the general demonstration with the extra details, but omitting those calculations that are very similar.

The main tool for the proof of Theorem \ref{theorem_linear} is a Carleman estimate for the adjoint system,
\begin{equation} \label{sistema_adj0}
\begin{cases}
-\phi_t=a_1  \phi_{xx}-b_1 \phi_x + c_1\phi +\lambda_1\int_0^1J_1(x-z)\phi(z,t)dz+\\
\qquad \qquad \qquad  \qquad \qquad \qquad \quad  \qquad \qquad \qquad  + \tilde a(x,t)\phi + \tilde c(x,t) \psi, \ \ \mbox{ in }  Q,\\
-\psi_t=a_2  \psi_{xx}-b_2\psi_x +c_2\psi +\lambda_2\int_0^1J_2(x-z)\psi(z,t)dz+\\
\qquad \qquad \qquad  \qquad \qquad \qquad \quad  \qquad \qquad \qquad  +\tilde b(x,t)\phi + \tilde d(x,t)\psi, \ \ \mbox{ in } Q,\\
\phi(0,t) = \phi(1,t)=\psi(0,t) = \psi(1,t)  = 0 ,\ \ \mbox{ in } 0<t<T,\\
\phi(x,T) = \phi_T(x), \psi(x,T) = \psi_T(x), \mbox{ in }  \Omega. 
\end{cases}
\end{equation}

This estimate allows us to establish an observability inequality. Then, we prove the null controllability of \eqref{eq:linearized_fixed_point} by using the observability.

See more details about Carleman estimate in the classical works of Fursikov and Imanuvilov \cite{Car12} and Fernández-Cara and Guerrero \cite{GlobalCarleman}, and specifically for our case in
\cite{Limaco-Lobosco-Yapu_24} and \cite{biccari2019}.

We need the following lemma of Fursikov and Imanuvilov \cite{Car12}.

\begin{lem}
Let $\omega \subset \subset (0,1)$ be a non empty open set. Then there exists a function $\eta_0 \in C^2([0,1])$ such that $\eta_0 > 0$ in the interval $(0,1)$, with $\eta_0 = 0$ at the boundary $\{0\} \cup \{1\}$ and $\vert (\eta_0)_x \vert >0$ in $[0,1]\setminus \omega$. 
\end{lem}

For some parameter $\kappa >0$, we define 
$$\sigma(x) = e^{4 \kappa \parallel \eta_0 \parallel_{\infty} } - e^{ \kappa (2\parallel \eta_0 \parallel_{\infty}+\eta_0(x))},$$ 
and the weight functions
$$ \alpha (x,t) = \frac{\sigma(x)}{t(T-t)} \quad \text{and} \quad \xi (x,t) = \frac{e^{\kappa (2\parallel \eta_0 \parallel_{\infty}+\eta_0(x))}}{t(T-t)}.$$

We also use the notations
\begin{equation}
\label{eq:alpha_sigma_mais_menos}
\begin{split}
\sigma^+ = \max_{x\in [0,1]}\sigma(x), \qquad \sigma^- = \min_{x\in [0,1]}\sigma(x), \\
\alpha^+(t) = \max_{x\in [0,1]}\alpha(x,t), \qquad \alpha^-(t) = \min_{x\in [0,1]}\alpha(x,t).
\end{split}
\end{equation}

We use the notation
\begin{align}\label{eq:formula_I}
I(z) := s^{-1}  \int_0^T\int_0^1 e^{-2s\alpha}\xi^{-1} (\vert z_t\vert^2 +\vert z_{xx}\vert^2)dxdt + s \kappa^2 \int_0^T\int_0^1 e^{-2s\alpha}\xi \vert z_x\vert^2dxdt + \nonumber \\ +s^3 \kappa^4 \int_0^T \int_0^1 e^{-2s \alpha}\xi^3 \vert z \vert^2 dx dt,        
\end{align} 
where $\alpha(x,t)$ and $\xi(x,t)$ are the Carleman weights defined above blowing-up at $t=0$ and $t=T$.

In this article we use the following Carleman inequality.

\begin{prop} \label{carleman_um_controle0}
Consider the adjoint system \eqref{sistema_adj0} with $\tilde c \in W^{2,\infty}(Q)$, $\tilde c(x,t) \neq 0$ in $\bar \omega \times (0,T)$  and   $\phi_T \in L^2(\Omega)$, $\psi_T \in L^2(\Omega)$. Moreover assume that the kernels $\lambda_i(t) J_i(x)$, $i=1,2$, satisfy \eqref{H_0} and suppose that the kernel $\lambda_1(t) J_1(x)$ satisfies \eqref{H2}. Then, there exist constants $C=C(\omega,a_1,a_2,b_1,c_1,T,\tilde c, \tilde a)$, $\kappa_1=C(\omega, M, b_1,b_2,c_1,c_2)$ and $s_1=C(\omega)(aT+ (aT)^2+\overline{K}^\frac{2}{3}T^2)$, where $\overline{K}$ was defined in \eqref{H_0} and $a = \max(a_1,a_2)$, such that the solution $\phi$, $\psi$ of the system \eqref{sistema_adj0} satisfy
\begin{equation}\label{Carleman2_coeff33}
I(\phi) + I(\psi) \leq Cs^7 \kappa^8\left(\iint_{\omega \times (0,T)} e^{-2s\alpha}\xi^7 \vert \phi \vert^2dxdt \right),
\end{equation}
for any constants $\kappa>\kappa_1$ and $s>s_1$.

\end{prop}

To prove this proposition, we will need some intermediate results.

We have the following Carleman estimate which is straightforwardly adapted from \cite{GlobalCarleman}. For more about Carleman's estimates, we refer to the important works \cite{Car12,Car22,Car32,Car42} and \cite{Car52}.

\begin{prop}[\cite{GlobalCarleman}]\label{prop2.2}
There exist positive constants $C=C(\Omega,\omega,a)$, $s_1  = C(\Omega,\omega)(aT+ (aT)^2)$ and $\kappa_1=C(\Omega,\omega)
$ such that, for any $s>s_1$, $\kappa > \kappa_1$,  $F \in L^2(Q)$, $z_T\in L^2(\Omega)$ and $a>0$ being a positive constant, the solution of the equation 
\begin{equation} \label{sistema4}
\begin{cases}
\displaystyle z_t+a\Delta z  = F, \ \ &\mbox{ in } \ \ 0 < x < 1,\text{ } 0<t<T, \\
z(0,t) = z(1,t) = 0 ,\ \ &\mbox{ in } 0<t<T,\\
\displaystyle z(x,T) = z_T(x), &\mbox{ in }  0 < x < 1,
\end{cases}
\end{equation}
satisfy
\begin{equation}\label{Carleman11}
I(z)\leq C \left[ s^3 \kappa^4 \iint_{\omega\times(0,T)}e^{-2s\alpha}\xi^3 \vert z \vert^2 dxdt + \int_0^T \int_0^1 e^{-2s\alpha} \vert F \vert^2 dx dt \right].
\end{equation}
\end{prop}
\begin{proof}

Set $\tau=a\,t$ and $w(x,\tau)=z(x,t)$. Then $z_t(x,t)=a\,w_{\tau}(x,\tau)$, $\Delta z(x,t)=\Delta w(x,\tau)$ and $adt=d\tau$. Equation \eqref{sistema4} becomes
\[
  w_{\tau}+\Delta w = \frac1a\,F\Bigl(x,\tfrac\tau a\Bigr) 
  \ \ \mbox{ in } \ \ 0 < x < 1,\text{ } 0<\tau<aT .
\]

In \cite{GlobalCarleman}, the inequality \eqref{Carleman11} holds for $w$ on the time interval $t\in(0,aT)$, with positive constants $C=C(\Omega,\omega)$, $s_1=C(\Omega,\omega)\,(aT+(aT)^2)$, and $\kappa_1=C(\Omega,\omega)$. In this inequality, the Carleman weights $\alpha$ and $\xi$ are defined on $(0,aT)$.
That is,
$$ \alpha (x,t) = \frac{\sigma(x)}{t(aT-t)} \qquad \text{and} \qquad \xi (x,t) = \frac{e^{\kappa (2\parallel \eta_0 \parallel_{\infty}+\eta_0(x))}}{t(aT-t)}.$$

The inequality \eqref{Carleman11} holds for $w$ in the following form. 
\begin{equation}\label{eq:carleman w}
 I_{aT}(w)\;\le\;C\!\left(
   s^{3}\lambda^{4}\!\!\iint_{\omega\times(0, aT)} e^{-2s\alpha}\xi^{3}|w|^{2} dxd\tau
   +\int_0^1\int_0^{aT}   e^{-2s\alpha}\left|\frac1a\,F\Bigl(x,\tfrac\tau a\Bigr)\right|^{2}dxd\tau \right).
\end{equation}

Here, we use  $I_{aT}(w)$ to specify that the terms are integrated in $\int_0^1\int_0^{aT}$ and not in $\int_0^1\int_0^{T}$ like in $I(z)$.  

 Using a simple change of variables $\tau = at$, we obtain de relation
\[
\min \left(\frac{1}{a},a \right) I(z)\leq I_{aT}(w),
\]
and the relation 
\begin{eqnarray*}
   s^{3}\lambda^{4}\!\!\iint_{\omega\times(0, aT)} e^{-2s\alpha}\xi^{3}|w|^{2} dxd\tau
   +\int_0^1\int_0^{aT}   e^{-2s\alpha}\left|\frac1a\,F\Bigl(x,\tfrac\tau a\Bigr)\right|^{2}dxd\tau  \\
\leq \max \left(\frac{1}{a},a \right)  \left(s^{3}\lambda^{4}\!\!\iint_{\omega\times(0, T)} e^{-2s\alpha}\xi^{3}|z|^{2} dxdt
   +\int_0^1\int_0^{T}   e^{-2s\alpha}\left|\,F\Bigl(x,t \Bigr)\right|^{2}dxdt \right).
\end{eqnarray*}

However, the coefficients $\alpha(x,\tau)$ and $\xi(x,\tau)$ that appears in the terms with the function $w$ are replaced by funcions $\alpha(x,at)$ and $\xi(x,at)$. Readapting these Carleman estimates and inserting these relations into \eqref{eq:carleman w} we obtain 
\begin{equation}
I(z)\leq C\max \left(\frac{1}{a^2},a^2 \right) \left[ s^3 \kappa^4 \int \int_{\omega\times(0,T)}e^{-2s\alpha}\xi^3 \vert z \vert^2 dxdt + \int_0^T \int_0^1 e^{-2s\alpha} \vert F \vert^2 dx dt \right].
\end{equation}
\end{proof}

In what follows we will use the following proposition.
\begin{prop}[\cite{biccari2019}] \label{lemaBiccari_0}
For any $k>0$, $s>1$ and $t \in (0,T)$, it holds that
\begin{equation}
exp\left( - \frac{(1+s)\sigma^-}{t(T-t)} \right)< exp\left( - \frac{s\sigma^+}{t(T-t)} \right), 
\end{equation}
where $\sigma^-$ and $\sigma^+$ are given in \eqref{eq:alpha_sigma_mais_menos}.
\end{prop}

We can apply Proposition \ref{prop2.2}, the estimate of Proposition \ref{lemaBiccari_0} and hypothesis \eqref{H_0} to obtain the following intermediate Carleman estimate for our linearized system.

\begin{prop}\label{Carleman1}
Let $\phi_T, \psi_T \in L^2(\Omega)$ be given and assume that the kernels $\lambda_i(t) J_i(x)$, $i=1,2$, satisfy \eqref{H_0}. Then, there exist constants  $C=C(\Omega,\omega,a_1,a_2)$, $\kappa_1=C(\Omega,\omega, M, b_1,b_2,c_1,c_2)$ and $s_1=C(\Omega,\omega)(aT+ (aT)^2+\overline{K}^\frac{2}{3}T^2)$, where $\overline{K}$ was defined in \eqref{H_0} and $a = \max(a_1,a_2)$, such that a solution $\phi$, $\psi$ of the system \eqref{sistema_adj0} with final conditions $\phi_T$ and  $\psi_T$, satisfy
\begin{equation}
\label{eq:carleman_dois_controles}
I(\phi) + I(\psi) \leq Cs^3 \kappa^4\left(\iint_{\omega \times (0,T)} \!  \!  \!  \!  \! e^{-2s\alpha}\xi^3 \vert \phi \vert^2dxdt+\iint_{\omega \times (0,T)} \!  \!  \!  \!  e^{-2s\alpha}\xi^3 \vert \psi \vert^2dxdt\right)
\end{equation}
for any constants $\kappa > \kappa_1$ and $s>s_1$.
\end{prop}

\begin{proof}

The proof is very similar to \cite{biccari2019} (part of Proposition 2.4) and \cite{Limaco-Lobosco-Yapu_24} (Proposition 2.3). We only need to observe that in the intermediate estimate 
$$I(\phi)+I(\psi)\leq C \left[ s^3 \kappa^4 \int \int_{\omega\times(0,T)} \!  \!  \!  \!  \!  \!  \! e^{-2s\alpha}\xi^3 \vert \phi \vert^2 dxdt +s^3 \kappa^4 \int \int_{\omega\times(0,T)} \!  \!  \!  \!  \!  \!  \! e^{-2s\alpha}\xi^3 \vert \psi \vert^2 dxdt\right.$$
$$\left.+\int_0^T \int_0^1 e^{-2s\alpha} \vert C_1 \phi_x + C_2\phi +C_3 \psi -\lambda_1(t) \int_0^1J_1(x-z)\phi(z,t)dz  \vert^2 dx dt \right. .$$
$$\left.+\int_0^T \int_0^1 e^{-2s\alpha} \vert C_4 \psi_x + C_5\psi +C_6 \phi -\lambda_2(t) \int_0^1J_2(x-z)\psi(z,t)dz  \vert^2 dx dt \right], $$
the coefficients of the system \eqref{sistema_adj0} appear as $C_1 = b_1$, $C_2 = -c_1 - \tilde a(x,t) $, $C_3 = -\tilde c (x,t)$, $C_4 = b_2$, $C_5 = -c_2 - \tilde d (x,t)$ and $C_6=-\tilde b (x,t) $.

We use that the function $\xi$ has a minimum value. Then, for $\kappa_1$ sufficiently big we get 
$$s \kappa^2\xi/2>Cb_1^2  \text{, } \qquad  s\kappa^2\xi/2>Cb_2^2,$$
$$s^3\kappa^4\xi^3/2>C(c_1 + \tilde a (x,t) )^2  \text{, } \qquad s^3\kappa^4\xi^3/2>C(\tilde c (x,t))^2,$$
$$s^3\kappa^4\xi^3/2>C(c_2 + \tilde d (x,t) )^2 \qquad \text{ and } \qquad s^3\kappa^4\xi^3/2>C(\tilde b (x,t))^2,$$
because $\tilde a$, $\tilde b$, $\tilde c$ e $\tilde d$ are bounded functions. 
In particular, we get the dependences of the constant $\kappa_1=C(\Omega, \omega, \tilde a, \tilde b, \tilde c, \tilde d, b_1,b_2,c_1,c_2) = C(\Omega, \omega,M,b_1,b_2,c_1,c_2)$.

The other steps of the proof are very similar to \cite{Limaco-Lobosco-Yapu_24} and we omit them.
\end{proof}

\begin{rem}\label{2_controles}
With this proposition we can easily obtain the null controllability with one control in each equation. For obtaining the same result using a single control in only one equation, we need the additional hypothesis \eqref{H2}.
\end{rem} 

\begin{proof}(Proposition \ref{carleman_um_controle0})

We adapt the steps of the demonstration in article \cite{Limaco-Lobosco-Yapu_24}, Proposition 2.4. 
The general method was developped by Gonzales-Burgos and de Teresa \cite{GB-Ter-10}.
Because of the non-constant coefficients, we see how some changes arise and the changes they imply in the hypothesis of our Proposition \ref{carleman_um_controle0}.

Using the last proposition applied to a set $\omega_0 \subset \omega $,
\begin{equation}\label{CarlemanPrim}
I(\phi) + I(\psi) \leq Cs^3 \kappa^4\left(\iint_{\omega_0 \times (0,T)} \! \! \! \! e^{-2s\alpha}\xi^3 \vert \phi \vert^2dxdt+\iint_{\omega_0 \times (0,T)} \! \! \! \! e^{-2s\alpha}\xi^3 \vert \psi \vert^2dxdt\right).
\end{equation}
Here, we have that $C = C(\Omega,\omega,a_1,a_2)$, $\kappa > \kappa_1$ and $s>s_1$ with $\kappa_1=C(\Omega,\omega, M, b_1,b_2,c_1,c_2)$ and $s_1=C(\Omega,\omega)(aT+ (aT)^2+\overline{K}^\frac{2}{3}T^2)$.

We analyse the last term. Let us consider a smooth function $\chi \in C^\infty_0(\omega)$ with $0 \leq \chi \leq 1$ in $\omega$ and $\chi \equiv 1$ in $\omega_0$. Thus 

\begin{align}\label{PrincipalM}
\iint_{\omega_0 \times (0,T)} e^{-2s\alpha} \kappa^4 s^3\xi^3 \vert \psi \vert^2dxdt \leq \iint_{\omega \times (0,T)} e^{-2s\alpha} \kappa^4 s^3\xi^3 \chi \vert \psi \vert^2dxdt \nonumber\\
= \iint_{\omega \times (0,T)} e^{-2s\alpha} \kappa^4 s^3\xi^3 \chi \psi \frac{1}{\tilde c} \left( -\phi_t - a_1  \phi_{xx} + b_1 \phi_x - c_1\phi -\right.\nonumber\\
- \lambda_1\int_0^1\left.J_1(x-z)\phi(z,t)dz - \tilde a \phi \right)       dxdt\nonumber\\
= M_1+M_2+M_3+M_4+M_5+M_6,
\end{align}
where the terms $M_i$, $i=1,...,6$, are given by
$$M_1= -\iint_{\omega \times (0,T)} e^{-2s\alpha} \kappa^4 s^3\xi^3 \chi   \frac{\psi}{\tilde c} \phi_t  dxdt$$

$$M_2= -\iint_{\omega \times (0,T)} e^{-2s\alpha} \kappa^4 s^3\xi^3 \chi   \frac{\psi}{\tilde c} a_1  \phi_{xx}   dxdt$$

$$M_3= \iint_{\omega \times (0,T)} e^{-2s\alpha} \kappa^4 s^3\xi^3 \chi   \frac{\psi}{\tilde c} b_1 \phi_x        dxdt$$

$$M_4= -\iint_{\omega \times (0,T)} e^{-2s\alpha} \kappa^4 s^3\xi^3 \chi  \frac{\psi}{\tilde c}  c_1\phi dxdt$$

$$M_5= - \iint_{\omega \times (0,T)} e^{-2s\alpha} \kappa^4 s^3\xi^3 \chi   \frac{\psi}{\tilde c} \lambda_1\int_0^1J_1(x-z)\phi(z,t)dz     dxdt$$

$$M_6= - \iint_{\omega \times (0,T)} e^{-2s\alpha} \kappa^4 s^3\xi^3 \chi  \frac{\psi}{\tilde c} \tilde a \phi dxdt.$$

Considering $\tilde c \in W^{2,\infty}(Q)$, $\tilde c(x,t) \neq 0$ in $\bar \omega \times (0,T)$, we have that the terms $M_i$, for $i= 1,2,...,6$, are well defined.

The term $M_3$ is estimated as follows:
\begin{align*}
    |M_3| &= \left | \iint_{\omega \times (0,T)} e^{-2s\alpha} \kappa^4 s^3\xi^3 \chi  \frac{\psi}{\tilde c} b_1 \phi_x        dxdt \right |  \\
    &\leq \iint_{\omega \times (0,T)} \left |( e^{-2s\alpha} \kappa^4 s^3\xi^3 \chi  \frac{\psi}{\tilde c} b_1 )_x \right | |\phi|  dxdt \\
    &\leq  \kappa^4 s^3 \iint_{\omega \times (0,T)} \biggl( e^{-2s\alpha}  (2s |\alpha_x|)\xi^3 \chi |\psi|\frac{|b_1|}{|\tilde c |} |\phi| + e^{-2s\alpha} 3 \xi^2 \xi_x \chi |\psi|\frac{|b_1|}{|\tilde c |}  |\phi| +  \\
    &  + e^{-2s\alpha} \xi^3 |\chi_x| |\psi| \frac{|b_1|}{|\tilde c |} |\phi| + e^{-2s\alpha} \xi^3 \chi  |\psi|\frac{|b_1 \tilde c_x|}{|\tilde c |^2}  |\phi| + e^{-2s\alpha} \xi^3 \chi |\psi_x| \frac{|b_1|}{|\tilde c |} |\phi| \biggr) dx dt\\
    &\leq C \kappa^4 s^3 \iint_{\omega \times (0,T)} \biggl( e^{-2s\alpha}  (2s |\alpha_x|)\xi^3 \chi |\psi| |\phi| + e^{-2s\alpha} 3 \xi^2 \xi_x \chi |\psi|  |\phi| +  \\
    &  + e^{-2s\alpha} \xi^3 |\chi_x| |\psi|  |\phi| + e^{-2s\alpha} \xi^3 \chi  |\psi|  |\phi| + e^{-2s\alpha} \xi^3 \chi |\psi_x| |\phi| \biggr) dx dt,
\end{align*}
where $C = C(b_1, \tilde c) = \max \{\frac{|b_1|}{|\tilde c |} ,\frac{|b_1 \tilde c_x|}{|\tilde c |^2} \}$. Because $\tilde c \in W^{2,\infty}(Q)$, $\tilde c(x,t) \neq 0$ in $\bar \omega \times (0,T)$, we have that the maximum is well defined.

Using the estimates $|\xi_x|\leq C(\omega) \kappa |\xi|$, $|\alpha_x| \leq C(\omega) \kappa |\xi|$ and $|\chi_x|\leq C(\omega)$, we get
\begin{align}\label{M3}
    |M_3| &\leq C \kappa^5 s^4 \iint_{\omega \times (0,T)} e^{-2s\alpha} \xi^4 |\psi| |\phi| dxdt + C k^5 s^3 \iint_{\omega \times (0,T)} e^{-2s\alpha} \xi^3 |\psi| |\phi| dxdt+ \nonumber\\
    & + C \kappa^4 s^3 \iint_{\omega \times (0,T)} e^{-2s\alpha} \xi^3 |\psi| |\phi| dxdt +\nonumber\\
     & + C \iint_{\omega \times (0,T)} e^{-2s\alpha} (s^{1/2} \xi^{1/2} k |\psi_x|) (k^3 s^{5/2} \xi^{5/2} |\phi|) dxdt.
\end{align}
Above, the constant $C$ has its dependence on $C = C(\omega, b_1, \tilde c)$.

The last term of the equation above (without the factor $C$) is estimated by 
\begin{align*}
 &  \iint_{\omega \times (0,T)} e^{-2s\alpha} (s^{1/2} \xi^{1/2} k |\psi_x|) (k^3 s^{5/2} \xi^{5/2} |\phi|) dxdt \\
  &\qquad \qquad \qquad \leq \varepsilon \iint_{\omega \times (0,T)} e^{-2s\alpha} \kappa^2 s \xi |\psi_x|^2 dxdt + C_\varepsilon \iint_{\omega \times (0,T)} e^{-2s\alpha} \kappa^6 s^5 \xi^5 |\phi|^2 dxdt.
\end{align*}

A similar estimate for the first three integrals in \eqref{M3} give
\begin{equation}
\label{eq:M3_eps}
\begin{split}
    |M_3| \leq &\varepsilon C \iint_{\omega \times (0,T)} e^{-2s\alpha} \kappa^4 s^3 \xi^3 |\psi|^2 dxd +\varepsilon C \iint_{\omega \times (0,T)} e^{-2s\alpha} \kappa^2 s \xi |\psi_x|^2 dxdt  + \\
    &+ C C_\varepsilon \iint_{\omega \times (0,T)} e^{-2s\alpha} \kappa^6 s^5 \xi^5 |\phi|^2 dxdt,
\end{split}
\end{equation}
with $C = C(\omega, b_1, \tilde c)$.

If we choose $\varepsilon_2>0$ such that $\varepsilon C \leq \varepsilon_2 $, then the first and second term of the right hand side of \eqref{eq:M3_eps} can be absorbed by $I(\psi)$ and we get 
\begin{equation}\label{M3.2}
|M_3|\leq \varepsilon_2  I(\psi) + C_{\varepsilon_2}   \iint_{\omega \times (0,T)} e^{-2s\alpha} \kappa^6 s^5 \xi^5 |\phi|^2 dxdt.    
\end{equation}
Here, $C_{\varepsilon_2} = C(\omega, b_1, \tilde c)$.

For the term $M_5$, we have
\begin{align*}
    |M_5| =& \left |\iint_{\omega \times (0,T)} e^{-2s\alpha} \kappa^4 s^3\xi^3 \chi  \frac{\psi}{\tilde c} \lambda_1(t)\int_0^1J_1(x-z)\phi(z,t)dz     dxdt \right | \\
    \leq& \varepsilon C  \iint_{\omega \times (0,T)} e^{-2s\alpha} \kappa^4 s^3 \xi^3 |\psi|^2 dxdt \\
    & +  C_\varepsilon \iint_{\omega \times (0,T)} e^{-2s\alpha} \kappa^4 s^3\xi^3 \left( \int_0^1 \lambda_1(t) J_1(x-z) \phi(z,t) dz  \right)^2 dxdt,
\end{align*}
with $C = C(\tilde c)$.

For an apropriate $\varepsilon_2>0$, the first term of the right hand side can be absorbed by $I(\psi)$.
\begin{align*}
    &|M_5| \leq  \varepsilon_2  I(\psi) \\
    & +  C_{\varepsilon_2} \iint_{\omega \times (0,T)} e^{-2s\alpha} \kappa^4 s^3\xi^3 \left( \int_0^1 \lambda_1(t) J_1(x-z) \phi(z,t) dz  \right)^2 dxdt. 
\end{align*}
By Cauchy-Schwarz inequality,
\begin{align}\label{M5}
    &|M_5| \leq  \varepsilon_2  I(\psi) \nonumber\\
    & +  C_{\varepsilon_2}  \iint_{\omega \times (0,T)} \! \! \! \! \! \! \! \! \! \! \! e^{-2s\alpha} \kappa^4 s^3 \xi^3 \left( \int_0^1 \lambda_1(t)^2 J_1(x-z)^2 dz \right) \left( \int_0^1 |\phi(z,t)|^2 dz \right) dxdt. 
\end{align}
Here, $C_{\varepsilon_2} = C( \tilde c)$.

%--------------------------------------------------------------------

Carleman weights are so that,
\begin{equation} \label{max_min}
    \alpha^{+}(t) < 2 \alpha^{-}(t).
\end{equation}
In fact, using the notations in \eqref{eq:alpha_sigma_mais_menos}, 
\begin{equation}
\begin{split}
 \alpha^{+}(t) = \max_{x\in [0,1]} \frac{\sigma(x)}{t(T-t)} < 2  \min_{x\in [0,1]} \frac{\sigma(x)}{t(T-t)} = 2 \alpha^{-}(t)
\end{split}
\end{equation}
with $\sigma(x) = e^{4k \parallel \eta_0 \parallel_{\infty} } - e^{k(2\parallel \eta_0 \parallel_{\infty}+\eta_0(x))}.$ 

So, equivalently,
\begin{equation}
\begin{split}
   \frac{e^{4k \parallel \eta_0 \parallel_{\infty} } - e^{2k\parallel \eta_0 \parallel_{\infty}}}{t(T-t)} < 2  \frac{e^{4k \parallel \eta_0 \parallel_{\infty} } - e^{3k\parallel \eta_0 \parallel_{\infty}}}{t(T-t)}. 
\end{split}
\end{equation}

This estimate is always true, since the following always holds
\begin{equation}
\begin{split}
   e^{4k \parallel \eta_0 \parallel_{\infty} } - e^{2k\parallel \eta_0 \parallel_{\infty}} < 2  (e^{4k \parallel \eta_0 \parallel_{\infty} } - e^{3k\parallel \eta_0 \parallel_{\infty}}).
\end{split}
\end{equation}
%and this estimate is always true.

Moreover, we denote 
$\xi^{-}(t) = min_{x\in[0,1]} \xi (x,t)$ and $\xi^{+}(t) = max_{x\in[0,1]} \xi (x,t)$.

Using the the conditions \eqref{H2} and \eqref{max_min} in \eqref{M5}, we have
\begin{align*}
    |M_5| &\leq \varepsilon_2 I(\psi) +  C_{\varepsilon_2} \iint_{\omega \times (0,T)} e^{-2s\alpha} \kappa^4 s^3 \xi^3 \delta e^{-2 s \alpha^{-}(t)} \int_0^1 |\phi(z,t)|^2 dz dxdt \\
    &\leq \varepsilon_2 I(\psi) + C_{\varepsilon_2} \delta \iint_{\omega \times (0,T)} e^{-2s\alpha^{-}(t)} \kappa^4 s^3 {\xi^{+}}^3(t) e^{- s \alpha^{+}(t)} \int_0^1 |\phi(z,t)|^2 dz dxdt \\
    & \leq \varepsilon_2  I(\psi) +  C_{\varepsilon_2} |\omega| \delta \int_0^T e^{-2s\alpha^{-}(t) - s \alpha^{+}(t) } \kappa^4 s^3 {\xi^{+}}^3(t) \int_0^1 |\phi(z,t)|^2 dz dt,
\end{align*}
where $|\omega|$ denotes the measure of $\omega$.

We want to absorb this by the second last term in $I(\phi)$, which is 
$$ \int_0^T \int_0^1  e^{-2 s\alpha(z,t)} \kappa^4 s^3 \xi^3 (z,t)  |\phi(z,t)|^2 dz dt. $$

In order to do that, it is enough to show that, for a small $\varepsilon_3>0$,
\begin{equation}
C_{\varepsilon_2} |\omega| \delta \int_0^T e^{-2s\alpha^{-}(t) - s \alpha^{+}(t) } \kappa^4 s^3 {\xi^{+}}^3(t) \int_0^1 |\phi(z,t)|^2 dz dt \leq   \varepsilon_3 \int_0^T \int_0^1  e^{-2 s\alpha(z,t)} \kappa^4 s^3 \xi^3 (z,t)  |\phi(z,t)|^2 dz dt.  
\end{equation}

In particular, this inequality can be obtained if we prove that
\begin{equation}
    \label{comparacao_absorver}
     C_{\varepsilon_2} |\omega| \delta \frac{ e^ {-2s\alpha^{-}(t) - s \alpha^{+}(t) } {\xi^{+}}^3(t) }{ e^{-2 s\alpha(z,t)} \xi^3 (z,t) } \leq \varepsilon_3,
\end{equation}
for $(z,t) \in [0,1] \times [0,T]$. 

First, $\frac{{\xi^{+}}^3(t)}{\xi^3 (z,t)} \leq \frac{{\xi^{+}}^3(t)}{{\xi^{-}}^3(t)} = \beta$, for the constant $\beta = e^{\kappa \parallel \eta_0 \parallel_{\infty} } > 1$, then \eqref{comparacao_absorver} is bounded by
\begin{equation*}
    %\label{comparacao_absorver}
     C_{\varepsilon_2} |\omega| \delta \frac{ e^ {-2s\alpha^{-}(t) - s \alpha^{+}(t) } {\xi^{+}}^3(t) }{ e^{-2 s\alpha(z,t)} \xi^3 (z,t) } \leq C_{\varepsilon_2} |\omega| \delta \beta \frac{ e^ {-2s\alpha^{-}(t) - s \alpha^{+}(t) } }{ e^{-2 s\alpha(z,t)}  } \leq \varepsilon_3,
\end{equation*}
Observe, using \eqref{max_min}, that 
$$-2\alpha^{-}(t) -  \alpha^{+}(t) +  2\alpha(z,t) < -2 \alpha^{+}(t) +  2\alpha(z,t) \leq 0,$$
then $ C_{\varepsilon_2} |\omega| \delta \beta e^ {-s (2\alpha^{-}(t) +  \alpha^{+}(t) -  2\alpha(z,t))} < C_{\varepsilon_2} |\omega| \delta \beta$, and we can choose $\delta$ depending only of $\omega$, $\tilde c$ and $\kappa$ such that  
\begin{equation}\label{M5.2}
    |M_5| \leq \varepsilon_2 I(\psi) + \varepsilon_3 I(\phi).
\end{equation}

The other terms we have the similar estimates as in \cite{sistema-CFLM}. For completeness we recall the main steps. We estimate the term $M_1$ as follows
$$\vert M_1 \vert  \leq  \iint_{\omega \times (0,T)} \vert (e^{-2s\alpha} (-2s \alpha_t) \kappa^4 s^3 \xi^3 \chi \frac{\psi}{\tilde c} \phi + e^{-2s\alpha} \kappa^4 s^3 3 \xi^2 \xi_t \chi \frac{\psi}{\tilde c} \phi +$$
$$- e^{-2s\alpha} \kappa^4 s^3 \xi^3 \chi \frac{\tilde c_t}{\tilde c^2} \psi \phi + e^{-2s\alpha} \kappa^4 s^3 \xi^3 \chi \frac{\psi_t}{\tilde c} \phi ) \vert dxdt .$$
Using that $|\alpha_t| \leq C(T) \xi^2$, $|\xi_t| \leq C(T) \xi^2$, $\frac{1}{\vert \tilde c(x,t) \vert} <C(\tilde c) $ and that $\frac{\vert \tilde c_t\vert }{\vert \tilde c^2\vert } < C(\tilde c)$
\begin{align}\label{M1}
|M_1| \leq & C\left( \iint_{\omega \times (0,T)} e^{-2s\alpha} \kappa^4 s^4 \xi^5 |\psi| |\phi| dxdt \right. +\nonumber\\
&+\left. \iint_{\omega \times (0,T)} e^{-2s\alpha} \kappa^4 s^3 \xi^3 |\psi_t| |\phi| dxdt \right) \nonumber\\
 \leq &\varepsilon I(\psi) + C_\varepsilon \iint_{\omega \times (0,T)} e^{-2s\alpha} \kappa^8 s^7 \xi^7 |\phi|^2 dxdt
\end{align}
with $C_\varepsilon = C(T,\tilde c)$.

For the term $M_2$, we use the identity
$$\left( e^{-2s\alpha } \xi^3 \chi \frac{\psi}{\tilde c}\right)_{xx} =  \left( e^{-2s\alpha } \xi^3 \chi\frac{1}{\tilde c}\right)_{xx} \psi + 2  \left( e^{-2s\alpha } \xi^3 \chi \frac{1}{\tilde c}\right)_x   \psi_x + e^{-2s\alpha } \xi^3 \chi \frac{1}{\tilde c}  \psi_{xx}.$$ 
We have the estimate
\[
\begin{aligned}
\left\vert \left( e^{-2s\alpha } \xi^3 \chi \frac{1}{\tilde c}\right)_{xx} \right\vert  &= \left\vert e^{-2s\alpha}\Bigl[\,
   4s^2\alpha_x^2\,\xi^3\frac{\chi}{\tilde c}
   \;-\;2s\,\alpha_{xx}\,\xi^3\frac{\chi}{\tilde c}
   \;-\;4s\,\alpha_x\Bigl(
       3\,\xi^2\xi_x\frac{\chi}{\tilde c}
       +\xi^3\bigl(\tfrac{\chi_x}{\tilde c}-\chi\tfrac{\tilde c_x}{\tilde c^2}\bigr)
     \Bigr) \right.\\
&\qquad\quad
   +3\bigl(2\,\xi\,\xi_x^2+\xi^2\,\xi_{xx}\bigr)\frac{\chi}{\tilde c}
   +6\,\xi^2\,\xi_x\Bigl(\tfrac{\chi_x}{\tilde c}-\chi\tfrac{\tilde c_x}{\tilde c^2}\Bigr)\\
&\qquad\quad
  \left. +\xi^3\Bigl[
      \tfrac{\chi_{xx}}{\tilde c}
      -2\,\tfrac{\chi_x\,\tilde c_x}{\tilde c^2}
      -\chi\Bigl(\tfrac{\tilde c_{xx}}{\tilde c^2}-2\,\tfrac{\tilde c_x^2}{\tilde c^3}\Bigr)
   \Bigr]
\,\Bigr] \right\vert \,\\
&\leq C e^{-2s\alpha } \kappa^2 s^2 \xi^5,
\end{aligned}
\]
where $C=C(\omega, \tilde c)$, because we use the estimates $|\xi_x|\leq C(\omega) \kappa |\xi|$, $|\alpha_x| \leq C(\omega) \kappa |\xi|$, $|\chi_x|\leq C(\omega)$, $|\alpha_{xx}| \leq C(\omega) \kappa^2 |\xi|$ and $|\xi_{xx}|\leq C(\omega) \kappa^2 |\xi|$.

Similarly, we have the estimate
$$ \left\vert \left( e^{-2s\alpha } \xi^3 \chi \frac{1}{\tilde c}\right)_{x} \right\vert \leq C e^{-2s\alpha } \kappa s \xi^4$$
with  $C=C(\omega, \tilde c)$.

Then, we obtain
\begin{align}\label{M2}
|M_2| &=  \left | \iint_{\omega \times (0,T)}  \left(  e^{-2s\alpha} \kappa^4 s^3\xi^3 \chi  \frac{\psi}{\tilde c} a_1 \right)_{xx}  \phi  dxdt  \right | \nonumber\\
&\leq C \iint_{\omega \times (0,T)} e^{-2s\alpha} \left( \kappa^6 (s\xi)^5 \vert \psi \vert + \kappa^5 (s \xi)^4 | \psi_x| + \kappa^4 (s\xi)^3 | \psi_{xx}|  \right) \phi dx dt \nonumber\\
&\leq \varepsilon I(\psi) + C_\varepsilon \iint_{\omega \times (0,T)} e^{-2s\alpha} \kappa^8 (s \xi)^7 |\phi|^2 dx dt.
\end{align}
Here, we have that $C_\varepsilon = C(\omega,\tilde c, a_1)$.

Finally, the terms $M_4$ and $M_6$ are estimated straightforwardly, e.g. 
\begin{equation}\label{M4}
\begin{split}
|M_4| \leq \varepsilon I(\psi) +  C_\varepsilon \iint_{\omega \times (0,T)} e^{-2s\alpha} \kappa^4 (s \xi)^3 |\phi|^2 dx dt,
\end{split}    
\end{equation}
with $C_\varepsilon = C(\tilde c, c_1)$.
Similarly for the term $M_6 = - \iint_{\omega \times (0,T)} e^{-2s\alpha} \kappa^4 s^3\xi^3 \chi \tilde a\frac{\psi}{\tilde c} \phi dxdt$ with $\chi$ a smooth function $\chi \in C^\infty_0(\omega)$ and $0 \leq \chi \leq 1$, we have
\begin{equation}\label{M6}
\begin{split}
|M_6| &\leq C \iint_{\omega \times (0,T)} e^{-2s\alpha} (\kappa^4 s^3\xi^3)^{1/2} (\kappa^4 s^3\xi^3)^{1/2} |\psi| |\phi| dxdt \\
&\leq \varepsilon I(\psi) + C_\varepsilon \iint_{\omega \times (0,T)} e^{-2s\alpha} \kappa^4 s^3\xi^3 |\phi|^2 dxdt,
\end{split}    
\end{equation}
with $C_\varepsilon = C(\tilde a, \tilde c)$.

Using \eqref{M1}, \eqref{M2}, \eqref{M3.2}, \eqref{M4}, \eqref{M5.2} and \eqref{M6} in \eqref{PrincipalM}, we finished the proof of the Carleman inequality \eqref{Carleman2_coeff33}.    
\end{proof}

As a consequence of the Carleman estimate, we prove now the observability inequality.
\begin{prop}\label{prop_observabilidade}
Consider the adjoint system \eqref{sistema_adj0} with $\tilde c \in W^{2,\infty}(Q)$, $\tilde c(x,t) \neq 0$ in $\bar \omega \times (0,T)$  and   $\phi_T \in L^2(\Omega)$, $\psi_T \in L^2(\Omega)$. Moreover, assume that the kernels $\lambda_i(t) J_i(x)$, $i=1,2$, satisfy \eqref{H_0} and suppose that the kernel $\lambda_1(t) J_1(x)$ satisfies \eqref{H2}. Then, there exist constants $C=C(M,\Omega,\omega,a_1,a_2,b_1,b_2,c_1,c_2,T,\tilde c, \overline{K})$, where $\overline{K}$ was defined in \eqref{H_0}, such that the solution $\phi$, $\psi$ of the system \eqref{sistema_adj0} satisfy    \begin{equation}\label{obsin}
        \parallel\phi(0)\parallel^2_{L^2(\Omega)} + \parallel\psi(0)\parallel^2_{L^2(\Omega)} \leq  C \iint_{\omega \times (0,T)} \vert \phi \vert^2 dxdt.
    \end{equation}
\end{prop}

\begin{proof}   
    Using Carleman inequality \eqref{Carleman2_coeff33} of Proposition  \ref{carleman_um_controle0} we have
    \begin{equation*}
         s^3 \kappa^4 \int_0^T\int_0^1 e^{-2s\alpha}\xi^3 ( \vert \phi \vert^2 + \vert \psi \vert^2) dxdt \leq C s^7 \kappa^8 \left(\iint_{\omega \times (0,T)} e^{-2s\alpha}\xi^7 \vert \phi \vert^2 dxdt \right),
    \end{equation*}
    and, fixing $\kappa>\kappa_1$, $s>s_1$,
    \begin{equation*}
         \int_0^T \int_0^1 e^{-2s\alpha}\xi^3 ( \vert \phi \vert^2 + \vert \psi \vert^2 ) dxdt \leq  C \iint_{\omega \times (0,T)} e^{-2s\alpha}\xi^7 \vert \phi \vert^2 dxdt.
    \end{equation*}
    
    The function $e^{-2s\alpha}\xi^m$ with $m \in \mathbb{N}$ is bounded by a positive constant $C_0$ from below in the interval $[T/4,3T/4]$, thus
    \begin{align} \label{obser1}
        C_0 \int_{T/4}^{3T/4} \int_0^1  ( \vert \phi \vert^2 + \vert \psi \vert^2 ) dxdt &\leq  C \iint_{\omega \times (0,T)} e^{-2s\alpha}\xi^7 \vert \phi \vert^2 dxdt \nonumber \\
        &\leq  C \iint_{\omega \times (0,T)} \vert \phi \vert^2 dxdt. 
    \end{align}
    
On the other hand, multiplying the first equation of the adjoint system \eqref{sistema_adj0} by $\phi$ and the second by $\psi$, integrating in $[0,1]$ and adding both equations we get
\begin{align*}
-\frac{1}{2} \frac{d}{dt} ( \| \phi \|^2 + \| \psi \|^2 ) &+ C_1( \|  \phi_x \|^2 + \|  \psi_x \|^2) \leq  C_2 (\| \phi \|^2 + \| \psi \|^2)+ \\
& + \int_0^1 \int_0^1 (\lambda_1(t) J_1(x-z) \phi(z,t) \phi(x,t)) dz dx \\
& + \int_0^1 \int_0^1 (\lambda_2(t) J_2(x-z) \psi(z,t) \psi(x,t) ) dz dx,
\end{align*}
where $\|\cdot\|$ denotes the $L^2$-norm on $\Omega$.

Using that the kernels are bounded, the last two integrals are bounded by a term of the form $C_3 \int_0^1 \int_0^1 \left( \phi(z,t)^2 + \phi(x,t)^2 + \psi(z,t)^2 + \psi(x,t)^2 \right) dz dx$, thus
\begin{equation*}
-\frac{1}{2} \frac{d}{dt} ( \| \phi \|^2 + \| \psi \|^2 ) + C_1( \|  \phi_x \|^2 + \|  \psi_x \|^2) \leq C_4 (\| \phi \|^2 + \| \psi \|^2), 
\end{equation*}
and
\begin{equation*}
-\frac{1}{2} \frac{d}{dt} ( \| \phi \|^2 + \| \psi \|^2 ) - C_4 ( \| \phi \|^2 + \| \psi \|^2) \leq 0. 
\end{equation*}

Using Gronwall's inequality we get
%$$ \int_0^1 ( \phi(x,0)^2 + \psi(x,0)^2 )dx \leq e^{2 C_4 T} \int_0^1 (\phi(x,t)^2 + \psi(x,t)^2) dx.
%$$
$$
\|\phi(0)\|^2 + \|\psi(0)\|^2 \leq e^{2 C_4 T} \left( \|\phi(t)\|^2 + \|\psi(t)\|^2 \right).
$$

From the last inequality and \eqref{obser1}, 
\begin{align*}
 \frac{T}{2} ( \| \phi(0) \|^2 + \| \psi(0) \|^2) &= \int_{T/4}^{3T/4} \left( \| \phi(0) \|^2 + \| \psi(0) \|^2 \right) dt\\
 &\leq e^{2 C_4 T} \int_{T/4}^{3T/4} \left( \| \phi(t) \|^2 + \| \psi(t) \|^2 \right) dt \\
 &\leq e^{2 C_4 T} C \iint_{\omega \times (0,T)} \vert \phi \vert^2 dxdt,    
\end{align*}
and this is the observability inequality.
\end{proof}

As a consequence of the observability, we obtain the null controllability of the linear system \eqref{eq:linearized_fixed_point} stated in Theorem \ref{theorem_linear}. The proof is like in \cite{Limaco-Lobosco-Yapu_24}. Because of this, we present here just a few extra details.

\begin{proof} \textbf{(Proof of Theorem \ref{theorem_linear}.)}  

For any $\varepsilon >0$ and any $\nu \in L^2(Q)$ we consider the functional
$$\mathcal{J}_\varepsilon(\nu) = \iint_{\omega \times (0,T)}  |\nu|^2dxdt + \frac{1}{\varepsilon} \int_0^1 (|y_\nu(T)|^2+|z_\nu(T)|^2)dx,$$
where $y_\nu$ and $z_\nu$ are solutions of \eqref{eq:linearized_fixed_point} associated to the control $\nu$ and fixed initial conditions $y_0$ and $z_0$.

It is not difficult to check that  $\mathcal{J}_\varepsilon$ is strictly convex, continuous and coercive in $L^2(Q)$, then it has a unique minimizer $\nu_\varepsilon \in L^2(Q)$. %We will give more details below.

\medskip
\noindent\textbf{1. $\mathcal{J}_\varepsilon$ is strictly convex.}

Let \(\nu_1,\nu_2\in L^2(Q)\) with $\nu_1 \neq \nu_2$ and fix \(\theta\in(0,1)\).  Define
\[
\nu_\theta \;=\;\theta\,\nu_1 + (1-\theta)\,\nu_2.
\]
By linearity of the system \eqref{eq:linearized_fixed_point} in the control \(\nu\), the corresponding solutions satisfy
\[
(y_{\nu_\theta},\,z_{\nu_\theta})
=\theta\,(y_{\nu_1},\,z_{\nu_1})
+(1-\theta)\,(y_{\nu_2},\,z_{\nu_2}).
\]

The value of \(\mathcal{J}_\varepsilon\) at \(\nu_\theta\) is
\begin{equation*}
\begin{aligned}
\mathcal{J}_\varepsilon(\nu_\theta)
&=\iint_{\omega \times (0,T)} \bigl|\theta\,\nu_1+(1-\theta)\,\nu_2\bigr|^2\,dx\,dt\\
&\quad+\;\frac1\varepsilon
\int_0^1\Bigl|\theta\,y_{\nu_1}(T)+(1-\theta)\,y_{\nu_2}(T)\Bigr|^2
+\Bigl|\theta\,z_{\nu_1}(T)+(1-\theta)\,z_{\nu_2}(T)\Bigr|^2
\,dx.
\end{aligned}
\end{equation*}

We have de strictly convexity of squares:
\[
|\theta a+(1-\theta)b|^2
< \theta|a|^2+(1-\theta)|b|^2
\]
for $a \neq b$.
Then, we obtain
\begin{align*}
\iint_{\omega \times (0,T)}|\theta\nu_1+(1-\theta)\nu_2|^2dxdt
&< \theta\iint_{\omega \times (0,T)}|\nu_1|^2dxdt
+(1-\theta)\iint_{\omega \times (0,T)}|\nu_2|^2dxdt,\\
\frac1\varepsilon\int_0^1|\theta\,y_1(T)+(1-\theta)\,y_2(T)|^2dx
&\leq\frac1\varepsilon\Bigl[
\theta\int_0^1|y_1(T)|^2dx
+(1-\theta)\int_0^1|y_2(T)|^2dx
\Bigr]
\end{align*}
and similarly for the \(z\)-term.

Then, we have that $\mathcal{J}_\varepsilon$ is strictly convex

\[
\mathcal{J}_\varepsilon(\theta\nu_1+(1-\theta)\nu_2)
<\theta\,\mathcal{J}_\varepsilon(\nu_1)
+(1-\theta)\,\mathcal{J}_\varepsilon(\nu_2).
\]

\medskip
\noindent\textbf{2. $\mathcal{J}_\varepsilon$ is continuous.}

Let \(\{\nu_{n}\}_{n\ge1}\subset L^{2}( \omega \times (0,T) ) \) be a sequence converging strongly to \(\nu\) in \(  L^{2}( \omega \times (0,T) ) \).

Since
\(\nu\mapsto \iint_{\omega \times (0,T)}|\nu|^{2}\,dx\,dt\) is continuous  on \(L^{2}( \omega \times (0,T) )\), we have
\[
\iint_{\omega \times (0,T)}|\nu|^{2}\,dx\,dt
=\lim_{n\to\infty}\iint_{\omega \times (0,T)}|\nu_{n}|^{2}\,dx\,dt.
\]

By standard energy estimates for the linear PDE system \eqref{eq:linearized_fixed_point}, there is a constant \(C>0\) such that for every control \(\nu\in L^{2}(\omega \times (0,T))\) and initial conditions $y_0$ and $z_0 \in L^2(0,1)$, we have
\[
\|y_{\nu}\|_{C([0,T];L^{2}(0,1))}
+\|z_{\nu}\|_{C([0,T];L^{2}(0,1))}
\;\le\;
C\,(\|\nu\|_{L^{2}(\omega \times (0,T))}+\| y_0 \|_{L^2(0,1)}+\| z_0 \|_{L^2(0,1)}).
\]

Hence
\[
y_{\nu_{n}}(T)\;\to\;y_{\nu}(T)
\quad\text{and}\quad
z_{\nu_{n}}(T)\;\to\;z_{\nu}(T)
\quad\text{strongly in }L^{2}(\Omega).
\]

In other words,
\[
\int_{\Omega}|y_{\nu}(T)|^{2}\,dx
=\lim_{n\to\infty}\int_{\Omega}|y_{\nu_{n}}(T)|^{2}\,dx,
\quad
\int_{\Omega}|z_{\nu}(T)|^{2}\,dx
=\lim_{n\to\infty}\int_{\Omega}|z_{\nu_{n}}(T)|^{2}\,dx.
\]

Putting together,
\[
\begin{aligned}
\mathcal{J}_\varepsilon(\nu)
&=\int_{Q}|\nu|^{2}\,dx\,dt
+\frac1\varepsilon\int_{\Omega}\bigl(|y_{\nu}(T)|^{2}+|z_{\nu}(T)|^{2}\bigr)\,dx\\
&=\lim_{n\to\infty}\int_{Q}|\nu_{n}|^{2}\,dx\,dt
\;+\;\frac1\varepsilon\lim_{n\to\infty}\int_{\Omega}\bigl(|y_{\nu_{n}}(T)|^{2}+|z_{\nu_{n}}(T)|^{2}\bigr)\,dx\\
&=\lim_{n\to\infty}\mathcal{J}_\varepsilon(\nu_{n}).
\end{aligned}
\]
This shows \(\mathcal{J}_\varepsilon\) is continuous.

\medskip
\noindent\textbf{3. $\mathcal{J}_\varepsilon$ is coercive.}

Coercivity is evident, due to the term $\iint_{\omega \times (0,T)}  |\nu|^2dxdt$.

Then, $\mathcal{J}_\varepsilon$ has a unique minimizer $\nu_\varepsilon \in L^2(\omega \times (0,T))$. We will denote $y_\varepsilon$ and $z_\varepsilon$ the associated solutions of \eqref{eq:linearized_fixed_point}. The Gâteaux derivative of the functional $\mathcal{J}_\varepsilon$ at the direction $\mu\in L^2(\omega \times (0,T))$ is given by
$$\mathcal{J}_\varepsilon'(\nu_\varepsilon)(\mu)= 2 \iint_{\omega \times (0,T)} \nu_\varepsilon \mu \, dxdt + \frac{2}{\varepsilon} \int_\Omega \left( y_\varepsilon(T) \hat y (T) + z_\varepsilon(T) \hat z(T) \right) dx,
$$
where $(\hat y, \hat z)$ are solutions of the system \eqref{eq:linearized_fixed_point} associated to $\nu  = \mu$ and initial conditions $y_0(x)=z_0(x)=0$. More precisely,

\begin{equation}\label{hat_y_z}
\begin{cases}
\hat y_t=a_1  \hat y_{xx}+b_1 \hat y_x + c_1 \hat y +\lambda_1\int_0^1 J_1(\zeta-x)\hat y(\zeta,t) d\zeta + \tilde a(x,t)\hat y + \tilde b(x,t)\hat z + \mu 1_\omega, \ \ \mbox{ in } Q,\\
\hat z_t=a_2  \hat z_{xx}+b_2 \hat z_x +c_2\hat z +\lambda_2\int_0^1J_2(\zeta-x)\hat z(\zeta,t)d\zeta+ \tilde c(x,t)\hat y + \tilde d(x,t)\hat z, \ \ \mbox{ in } Q,\\
\hat y(0,t) = \hat y(1,t)=\hat z(0,t) = \hat z(1,t)  = 0 ,\ \ 
\mbox{ for } 0<t<T,\\
\hat y(x,0) = 0, \hat z(x,0) = 0, \mbox{ in }  \Omega.
\end{cases}
\end{equation}

If the minimum is reached at $\nu_\varepsilon$, the Gateaux derivative at this point is zero for any $\mu \in L^2(\omega \times (0,T))$.

\begin{equation}\label{eq:funcional_derivada_direc_nula}
0= 2 \iint_{\omega \times (0,T)} \nu_\varepsilon \mu \, dxdt + \frac{2}{\varepsilon} \int_\Omega \left( y_\varepsilon(T) \hat y (T) + z_\varepsilon(T) \hat z(T) \right) dx.    
\end{equation}

We use the adjoint equation \eqref{sistema_adj0} with final conditions
\begin{equation}\label{final_sistema_acessorio} 
\phi_T = -\frac{1}{\varepsilon} u_\varepsilon(T), \quad \psi_T = -\frac{1}{\varepsilon} v_\varepsilon (T).
\end{equation}
For simplicity, we will call the associated solution $(\phi, \psi)$.

Multiplying the first equation of \eqref{hat_y_z} by $\phi$, the second by $\psi$, integrating in $Q$ and adding both equations, we have
\begin{align*}
&\int_Q \left( \hat y_t-a_1  \hat y_{xx}-b_1 \hat y_x - c_1 \hat y -\lambda_1\int_0^1 J_1(\zeta-x)\hat y(\zeta,t) d\zeta - \tilde a(x,t)\hat y - \tilde b(x,t)\hat z \right) \phi \, dxdt +\\
&+ \int_Q \left( \hat z_t-a_2  \hat z_{xx}-b_2 \hat z_x -c_2\hat z -\lambda_2\int_0^1J_2(\zeta-x)\hat z(\zeta,t)d\zeta- \tilde c(x,t)\hat y - \tilde d(x,t)\hat z \right)\psi \, dxdt \\
&= \int_Q \mu \phi 1_\omega \, dxdt.
\end{align*}

Integrating by parts and using that $\phi$ and $\psi$ verify the adjoint system \eqref{sistema_adj0} and the final conditions \eqref{final_sistema_acessorio},

$$\int_\Omega \hat y(T) (-\frac{1}{\varepsilon} y_\varepsilon (T)) \, dx+ \int_\Omega \hat z(T) (-\frac{1}{\varepsilon} z_\varepsilon (T)) \, dx = \int_Q \mu \phi_\varepsilon 1_\omega \, dxdt.$$

Thus, using \eqref{eq:funcional_derivada_direc_nula}, 
$$\iint_{\omega \times (0,T)} \nu_\varepsilon \mu \, dxdt = \iint_{\omega \times (0,T)}  \mu \phi_\varepsilon 1_\omega \, dxdt, $$
for any $\mu \in L^2(\omega \times (0,T))$. This implies that 
\begin{equation} \label{eq:formulacontrol}
\nu_\varepsilon = \phi_\varepsilon 1_\omega.
\end{equation}

On the other hand, considering the linearized system \eqref{eq:linearized_fixed_point} associated with $y_\varepsilon$, $z_\varepsilon$ and $\nu_\varepsilon$,  multiplying the first equation of \eqref{eq:linearized_fixed_point} by $\phi_\varepsilon$, the second equation of \eqref{eq:linearized_fixed_point} by $\psi_\varepsilon$, integrating in $Q$, adding and using \eqref{eq:formulacontrol}, we have
\begin{align*}
&\int_Q \left(  y_{\varepsilon,t} - a_1  y_{\varepsilon,xx} - b_1 y_{\varepsilon,x} - c_1y_\varepsilon - \right. \\
&\left. - \lambda_1\int_0^1J_1(\zeta-x)y_\varepsilon (\zeta,t)dz - \tilde a(x,t) y_\varepsilon - \tilde b(x,t) z_\varepsilon \right) \phi_\varepsilon \, dxdt +\\ 
&+ \int_Q \left(  z_{\varepsilon,t} - a_2  z_{\varepsilon,xx} - b_2z_{\varepsilon,x} - c_2z_\varepsilon -\right.\\
&\left. - \lambda_2\int_0^1J_2(\zeta-x)z_\varepsilon (\zeta,t)dz - \tilde c(x,t)y_\varepsilon - \tilde d(x,t)z_\varepsilon \right) \psi_\varepsilon \, dxdt = \\
& = \int_Q |\phi_\varepsilon|^2 1_\omega \, dxdt.
\end{align*}

Integrating by parts,  we obtain
$$(y_\varepsilon, \phi_\varepsilon)_{L^2(Q)} \vert_0^T + (z_\varepsilon,\psi_\varepsilon)_{L^2(Q)} \vert_0^T = \int_{\omega \times (0,T)} |\phi_\varepsilon|^2 \, dxdt.$$

Thus, replacing the initial and final conditions,
$$ \int_{\omega \times (0,T)} |\phi_\varepsilon|^2 \, dxdt + \frac{1}{\varepsilon} \int_\Omega |y_\varepsilon(T)|^2  \, dx + \frac{1}{\varepsilon} \int_\Omega |z_\varepsilon(T)|^2 \, dx =-(y_0,\phi_\varepsilon(0)) - (z_0,\psi_\varepsilon(0)).$$

That is, 
$$\int_{\omega \times (0,T)} |\phi_\varepsilon|^2 dxdt + \frac{1}{\varepsilon} \int_\Omega |y_\varepsilon(T)|^2 dx + \frac{1}{\varepsilon} \int_\Omega |z_\varepsilon(T)|^2 dx \leq$$
$$\leq (\Vert y_0 \Vert ^2+ \Vert z_0 \Vert^2)^{1/2} ( \Vert \phi_\varepsilon(0) \Vert^2 + \Vert \psi_\varepsilon(0) \Vert^2 )^{1/2}.$$

Using the observability inequality, and the elementary inequality $ab \leq \varepsilon a^2+\frac{b^2}{4\varepsilon}$, we have
$$ \int_{\omega \times (0,T)} |\phi_\varepsilon|^2 \, dxdt+ \frac{1}{\varepsilon} \int_\Omega |y_\varepsilon(T)|^2 \, dx + \frac{1}{\varepsilon} \int_\Omega |z_\varepsilon(T)|^2 \, dx \leq$$
$$\leq \frac{C}{2} (\Vert y_0\Vert^2+\Vert z_0\Vert^2) +  \frac{1}{2} \int_{\omega \times (0,T)} |\phi_\varepsilon|^2 \, dxdt.$$

Thus, using \eqref{eq:formulacontrol},
\begin{equation} \label{eq:limitacao_controle}
\frac{1}{2} \int_{\omega \times (0,T)} |\nu_\varepsilon|^2 \, dxdt + \frac{1}{\varepsilon} \int_\Omega |y_\varepsilon(T)|^2 \, dx + \frac{1}{\varepsilon} \int_\Omega |z_\varepsilon(T)|^2  \, dx\leq \frac{C}{2} (\Vert y_0\Vert^2+\Vert z_0\Vert^2).
\end{equation}

Kakutani's compactness theorem implies that there exists a subnet
of $\nu_\varepsilon$, with $\varepsilon \rightarrow 0$, denoted with the same symbol, such that
$\nu_\varepsilon \rightharpoonup \nu$ (weakly) in $L^2(\omega \times (0,T))$.

We use the energy estimates
\begin{equation*}
\begin{split}
\Vert  y_\varepsilon \Vert_{L^2(0,T,H_0^1(\Omega))} + &\Vert  z_\varepsilon \Vert_{L^2(0,T,H_0^1(\Omega))} + \Vert y_{\varepsilon,t} \Vert_{L^2(0,T,H^{-1} (\Omega))} + \Vert  z_{\varepsilon,t} \Vert_{L^2(0,T,H^{-1} (\Omega))} \\
&\leq C_1 ( \Vert  y_0 \Vert_{L^2(\Omega)}+\Vert  z_0\Vert_{L^2(\Omega)}+ \Vert  \nu_\varepsilon \Vert_{L^2(Q)} ), 
\end{split}
\end{equation*}
and, using \eqref{eq:limitacao_controle}, we have
\begin{equation}
\label{eq:energy_est_L2}
\begin{split}
\Vert  y_\varepsilon \Vert_{L^2(0,T,H_0^1(\Omega))} + &\Vert  z_\varepsilon \Vert_{L^2(0,T,H_0^1(\Omega))} + \Vert y_{\varepsilon,t} \Vert_{L^2(0,T,H^{-1} (\Omega))} + \Vert  z_{\varepsilon,t} \Vert_{L^2(0,T,H^{-1} (\Omega))} \\
&\leq C_2 ( \Vert  y_0 \Vert_{L^2(\Omega)}+\Vert  z_0\Vert_{L^2(\Omega)} ).
\end{split}
\end{equation}
%$$\Vert  u_\varepsilon \Vert_{L^2(0,T,H_0^1(\Omega))} + \Vert  v_\varepsilon \Vert_{L^2(0,T,H_0^1(\Omega))} + \Vert u_{\varepsilon,t} \Vert_{L^2(0,T,H^{-1} (\Omega))} + \Vert  v_{\varepsilon,t} \Vert_{L^2(0,T,H^{-1} (\Omega))} $$
%$$\leq C_2 ( \Vert  u_0 \Vert_{L^2(\Omega)}+\Vert  v_0\Vert_{L^2(\Omega)} ).$$

Thus, we get the weak convergence of the subnets $y_{\varepsilon} \rightharpoonup y$, $ z_{\varepsilon} \rightharpoonup z$ in $L^2(0,T,H_0^1(\Omega))$ and
$y_{\varepsilon,t} \rightharpoonup y_t$, $z_{\varepsilon,t} \rightharpoonup z_t$ in $L^2(0,T,H^{-1}(\Omega))$. 

Moreover, \eqref{eq:energy_est_L2} implies by \cite{Evans} (Sec 5.9, Thm 3) that, after a possible redefinition on a set of measure zero, $y_\varepsilon$, $z_\varepsilon$ belong to $C([0,T],L^2(\Omega))$ and
\begin{equation*}
%\label{eq:energy_est_cont}
\Vert y_\varepsilon \Vert_{C([0,T],L^2(\Omega)} + \Vert z_\varepsilon \Vert_{C([0,T],L^2(\Omega)} \leq C_3 ( \Vert y_0 \Vert_{L^2(\Omega)}+ \Vert z_0 \Vert_{L^2(\Omega)}).
\end{equation*}

Let us recall the following result:
\begin{proposition}\cite{breziscazenave}(Prop A.2.46(i))
Let $X \hookrightarrow Y$ be two Banach spaces and let $f_n$ be a bounded sequence of $L^\infty(I,X) \cap W^{1,r}(I,Y)$, for some $r>1$. If $X$ is reflexive and if $I$ is bounded, then the following hold:
there exists $f \in L^\infty(I,X)$, $f : \bar I \to X$ being weakly continuous, and a subsequence $n_k$ such that $f_{n_k}(t) \rightharpoonup f(t)$ as $k \to \infty$, for every $t \in \bar I$. In particular,
$$
\int_I f_n(t) \phi(t)dt \rightharpoonup \int_I f(t)\phi(t) dt
$$
as $n \to \infty$ for every $\phi \in C_c(I)$.
\end{proposition}

Using this for $X = L^2(\Omega)$, $Y=H^{-1}(\Omega)$ and  $r=2$, we get the existence of $y \in L^\infty(0,T,L^2(\Omega))$ and $z \in L^\infty(0,T,L^2(\Omega))$ such that, for a subnet,     
$y_\varepsilon (t) \rightharpoonup y(t)$ and $z_\varepsilon(t) \rightharpoonup z(t)$ in $L^2(\Omega)$ for every $t \in [0,T]$. In particular we have that $y_\varepsilon(T) \rightharpoonup y(T)$ and $z_\varepsilon(T) \rightharpoonup z(T)$ weakly in $L^2(\Omega)$.   

\color{black}
As a consequence, there exist subnets such that
\begin{align} \label{eq:weaklimits}
& y_{\varepsilon,t} \rightharpoonup y_t, \quad z_{\varepsilon,t} \rightharpoonup z_t, \qquad \text{ in } L^2(0,T,H^{-1}(\Omega)) \\
& y_{\varepsilon} \rightharpoonup y, \quad z_{\varepsilon} \rightharpoonup z, \qquad \text{ in } L^2(0,T,H_0^1(\Omega)) \\
& y_\varepsilon(T) \rightharpoonup y(T), \quad z_\varepsilon(T) \rightharpoonup z(T), \qquad \text{ in } L^2(\Omega),
\end{align}
and from  \eqref{eq:limitacao_controle},
\begin{equation} \label{eq:stronglimit}
y_\varepsilon(T) \rightarrow 0, \quad z_\varepsilon(T) \rightarrow 0, \quad \text{ in } L^2(\Omega), \text{ because } \varepsilon \to 0.
\end{equation}

From unicity of the weak limit, we have $u(T)=0$ and $v(T)=0$. 

Using \eqref{eq:weaklimits}-\eqref{eq:stronglimit}, we can pass to the limit in the weak form of the approximate system. Thus, the limit functions $y$, $z$ satisfy the weak form of the original system \eqref{eq:linearized_fixed_point}, with $y(T)=0$ and $z(T)=0$ and with $y, \ z \in L^2(0,T,H^1_0(\Omega))$ and $y_t, \ z_t \in L^2(0,T,H^{-1}(\Omega))$.

The estimate \eqref{ControlEstimate} follows immediately from \eqref{eq:limitacao_controle}.
\end{proof}

\section{Proof of Theorem \ref{maintheorem0}}
\label{sec:proof_nonlinear0}

To prove the main theorem we will use the following version of Kakutani's theorem: 
\begin{teo}\label{teo:kaku}(Kakutani's fixed point theorem) 
Let $\X$ be a complete Hausdorff locally convex linear topological space. Let $B \subset \X$ be nonempty, convex and compact. Let $\Phi: B \to 2^{B}$ be a multivalued map that has closed graph and such that for any $(\bar y, \bar z) \in B$, $\Phi(\bar y,\bar z)$ is nonempty, convex and compact. Then, the set of fixed points of $\Phi$ is nonempty.
\end{teo}
This theorem can be easily derived from the Theorem 11.9 (page 138) of the book \cite{Kakutani}. In our context, the graph of the multivalued map $\Phi$ is defined by
$$
Graph(\Phi) = \left\{ \left( (\bar{y}, \bar{z}), (u, v) \right) \in B \times B : (u, v) \in \Phi(\bar{y}, \bar{z}) \right\}.
$$

In the present article the specific functional spaces are defined as follows:
\begin{equation}\label{def:spaces}
\X:=L^{2}(Q)\times L^{2}(Q),\qquad 
W':=\bigl\{u\in L^{2}\bigl(0,T;H^{1}_{0}(\Omega)\bigr):\,u_t\in L^{2}\bigl(0,T;H^{-1}(\Omega)\bigr)\bigr\},
\end{equation}
and denote by $\|\cdot\|_{W'}$ the natural norm of $W'$. 
$$   \|u\|_{W'}^2 := \|u\|_{L^2(0,T;H_0^1(\Omega))}^2 + \|u_t\|_{L^2(0,T;H^{-1}(\Omega))}^2.
$$

We also introduce the ball
\begin{equation}\label{def:B_R}
B:=B_{W'}(0,R):=\bigl\{(y,z)\in W'\times W':\|y\|_{W'}^{2}+\|z\|_{W'}^{2}\le R^{2}\bigr\},
\end{equation}
where $R>0$ will be chosen later and the norm of $W' \times W'$ is defined by
\[
\|(y,z)\|_{W' \times W'}^{2}
  \;=\;
    \|y\|_{L^{2}(0,T;H^{1}_{0}(\Omega))}^{2}
  + \|y_{t}\|_{L^{2}(0,T;H^{-1}(\Omega))}^{2}
  + \|z\|_{L^{2}(0,T;H^{1}_{0}(\Omega))}^{2}
  + \|z_{t}\|_{L^{2}(0,T;H^{-1}(\Omega))}^{2}.
\]

%\color{black}

The multivaluated map is defined on $B \subset \X$ by
$$\Phi: B \to 2^{B},$$
such that for fixed initial conditions in system \eqref{eq:linearized_fixed_point} and for the constant $C$ of estimate \eqref{ControlEstimate}, we have:
\begin{equation}
\begin{split}
\Phi(\bar y,\bar z) = \{ & (y,z) \in \X \vert \text{ there exists }  \nu\in L^2(\omega\times(0,T)) \text{ such that } (y,z) \text{ is solution of } \eqref{eq:linearized_fixed_point}\\
&\text{ with }  \ y(T)=0, \ z(T)=0, \  \|\nu\|_{L^2(\omega\times(0,T))} \leq C (\|y_0\|_{L^2(\Omega)} + \|z_0\|_{L^2(\Omega)}) \}.
\end{split}
\end{equation}

Here, it is crucial that the constant $C$ does not depend on the element $(\bar y,\bar z)$ since this uniformity of the constant will be crucial in the proof.

If $(y,z)$ is fixed point of $\Phi$, i.e. $(y,z) \in \Phi(y,z)$, then $(y,z)$ is a solution of
\begin{equation}
\label{system_fixed_point}
\begin{cases}
y_t=a_1  y_{xx}+b_1 y_x + c_1 y +\lambda_1\int_0^1 J_1(\zeta-x)y(\zeta,t) d\zeta + a(y,z)y + b(y,z)z + \nu 1_\omega, \ \ \mbox{ in } Q,\\
z_t=a_2  z_{xx}+b_2 z_x +c_2z +\lambda_2\int_0^1J_2(\zeta-x)z(\zeta,t)d\zeta+ c(y,z)y + d(y,z)z, \ \ \mbox{ in } Q,\\
y(0,t) = y(1,t)=z(0,t) = z(1,t)  = 0 ,\ \ 
\mbox{ for } 0<t<T,\\
y(x,0) = y_0(x), z(x,0) = z_0(x), \mbox{ in }  \Omega,
\end{cases}
\end{equation}
which, from the decomposition \eqref{eq:decomp_F_G}, is our original system \eqref{sistema0}. Therefore, the system will be null controllable.

\begin{proof}(Theorem \ref{maintheorem0})

Since $H_0^1(\Omega) \subset L^2(\Omega)$ compactly, and $L^2(\Omega) \subset H^{-1}(\Omega)$ continuously, then by the theorem of Aubin-Lions, $W' \subset L^2(0,T,L^2(\Omega))$ compactly. 
As a product space, we also get that $ W'\times W' \subset L^2(Q) \times L^2(Q)$ compactly. Thus, $\mathbf{B}$ \textbf{is compact in} $\X$, i.e. if $(y_n,z_n) \in B$ is a bounded sequence then there exist a subsequence, denoted also by $(y_n,z_n)$, such that
$$
(y_n,z_n) \to (y,z) \qquad \text{in} \qquad  \X.
$$ 

Furthermore $\mathbf{B}$ \textbf{ is convex.} Indeed, if $(y_1,z_1)$ and $(y_2,z_2)$ belong to $B$, we have
$$
\|(y_1,z_1)\|_{W'\times W'} \leq R, \qquad \|(y_2,z_2)\|_{W'\times W'} \leq R, 
$$
and then, by the properties of the norm,
$$
\|\lambda(y_1,z_1) + (1-\lambda)(y_2,z_2)\|_{W'\times W'} \leq \lambda \|(y_1,z_1)\|_{W'\times W'} + (1-\lambda) \|(y_2,z_2)\|_{W'\times W'} \leq R.
$$

Obviously $\mathbf{B}$ \textbf{is non-empty.}

Now, we prove that for any $(\bar y, \bar z) \in B$ the set $\Phi(\bar y, \bar z)$ is convex.

Let $(y_1,z_1)$ and $(y_2,z_2)$ belong to $\Phi(\bar y, \bar z)$. Thus, for $i=1,2$,
\begin{equation*}
\begin{cases}
y_{i,t}=a_1  y_{i,xx}+b_1 y_{i,x} + c_i y_i +\lambda_1\int_0^1 J_1(\zeta-x)y_i(\zeta,t) d\zeta + a(\bar y,\bar z)y_i + b(\bar y,\bar z)z_i + \nu_i 1_\omega, \ \ &\mbox{ in } Q,\\
z_{i,t}=a_2  z_{i,xx}+b_2 z_{i,x} +c_2 z_i +\lambda_2\int_0^1J_2(\zeta-x)z_i(\zeta,t)d\zeta+ c(\bar y,\bar z)y_i + d(\bar y,\bar z)z_i , \ \ &\mbox{ in } Q,\\
y_i(0,t) = y_i(1,t)=z_i(0,t) = z_i(1,t)  = 0 ,\ \ 
\mbox{ for } 0<t<T,\\
y_i(x,0) = y_0(x), z_i(x,0) = z_0(x), \mbox{ in }  \Omega,
\end{cases}
\end{equation*}
and
$$
y_i(T)=0, \qquad z_i(T)=0,
$$
$$
\nu_i \in L^2(\omega\times(0,T)), \qquad \|\nu_i\|  \leq C (\|y_0\|_{L^2(\Omega)} + \|z_0\|_{L^2(\Omega)}).
$$

The convex combination $(y,z) = \lambda(y_1,z_1)+(1-\lambda)(y_2,z_2)$ is a solution of
\begin{equation*}
\begin{cases}
y_t=a_1  y_{xx}+b_1 y_x + c_1 y +\lambda_1\int_0^1 J_1(\zeta-x)y(\zeta,t) d\zeta + a(\bar y,\bar z)y + b(\bar y,\bar z)z = (\lambda \nu_1 + (1-\lambda) \nu_2) 1_\omega, \ \ &\mbox{ in } Q,\\
z_t=a_2  z_{xx}+b_2 z_x +c_2z +\lambda_2\int_0^1J_2(\zeta-x)z(\zeta,t)d\zeta+ c(\bar y,\bar z)y + d(\bar y,\bar z)z = 0, \ \ &\mbox{ in } Q,\\
y(0,t) = y(1,t)=z(0,t) = z(1,t)  = 0 ,\ \ 
\mbox{ for } 0<t<T,\\
y(x,0) = y_0(x), z(x,0) = z_0(x), \mbox{ in }  \Omega,
\end{cases}
\end{equation*}
and
$$
y(T)=0 \qquad \text{and} \qquad z(T)=0.
$$

Taking $\nu=\lambda \nu_1 + (1-\lambda)\nu_2 \in L^2(\omega\times (0,T))$, we have
$$
\|\nu\|_{L^2(\omega\times (0,T))} \leq \lambda\|\nu_1\|+(1-\lambda)\|\nu_2\|  \leq C (\|y_0\|_{L^2(\Omega)} + \|z_0\|_{L^2(\Omega)}).
$$
This implies that $\mathbf{\Phi(\bar y,\bar z)}$ \textbf{is convex.}

Now we show that for every $(\bar y, \bar z) \in B$ we have that  $\Phi(\bar y,\bar z)$ is compact. Let $(y_n,z_n)$ denote a sequence in $\Phi(\bar y,\bar z)$. Then, $(y_n,z_n)$ is a solution of
\begin{equation}\label{sistema3}
\begin{cases}
y_{n,t}=a_1  y_{n,xx}+b_1 y_{n,x} + c_i y_n +\lambda_1\int_0^1 J_1(\zeta-x)y_n(\zeta,t) d\zeta + a(\bar y,\bar z)y_n + b(\bar y,\bar z)z_n + \nu_n 1_\omega,   &\mbox{ in } Q,\\
z_{n,t}=a_2  z_{n,xx}+b_2 z_{n,x} +c_2 z_n +\lambda_2\int_0^1J_2(\zeta-x)z_n(\zeta,t)d\zeta+ c(\bar y,\bar z)y_n + d(\bar y,\bar z)z_n ,   &\mbox{ in } Q,\\
y_n(0,t) = y_n(1,t)=z_n(0,t) = z_n(1,t)  = 0 ,\ \ 
\mbox{ for } 0<t<T,\\
y_n(x,0) = y_0(x), z_n(x,0) = z_0(x), \mbox{ in }  \Omega,
\end{cases}
\end{equation}
and
$$
y_n(T)=0 \qquad \text{, } \qquad z_n(T)=0,
$$
$$
\nu_n \in L^2(\omega\times(0,T)), \qquad \|\nu_n\| \leq C (\|y_0\|_{L^2(\Omega)} + \|z_0\|_{L^2(\Omega)}).
$$

We have to show that there exists a subsequence, denoted also $(y_n,z_n)$ such that $(y_n,z_n) \to (y,z)$ strongly in $\X$.

By the estimate of $\nu_n$ and by the energy estimates like \eqref{eq:energy_est_L2} of the linear system \eqref{sistema3}, there is a subsequence, denoted also $(y_n,z_n)$, such that
\begin{equation}
\label{eq:subseqs0}
\begin{split}
y_n \rightharpoonup y, \quad z_n \rightharpoonup z, \quad &\text{weakly in} \quad L^2(0,T,H_0^1(\Omega)),\\
y_{n,t} \rightharpoonup y_t, \quad z_{n,t} \rightharpoonup z_t, \quad &\text{weakly in} \quad L^2(0,T,H^{-1}(\Omega)),\\
\nu_n \rightharpoonup \nu, \qquad &\text{weakly in} \quad L^2(\omega\times (0,T)).
\end{split}
\end{equation}

With an argument analogous to the proof of Theorem \ref{theorem_linear}, we can pass to the limit in \eqref{sistema3} and get
\begin{equation*}
\begin{cases}
y_t=a_1  y_{xx}+b_1 y_x + c_1 y +\lambda_1\int_0^1 J_1(\zeta-x)y(\zeta,t) d\zeta + a(\bar y,\bar z)y + b(\bar y,\bar z)z + \nu 1_\omega, \ \ &\mbox{ in } Q,\\
z_t=a_2  z_{xx}+b_2 z_x +c_2 z +\lambda_2\int_0^1J_2(\zeta-x)z(\zeta,t)d\zeta+ c(\bar y,\bar z)y + d(\bar y,\bar z)z , \ \ &\mbox{ in } Q,\\
y(0,t) = y(1,t)=z(0,t) = z(1,t)  = 0 ,\ \ 
\mbox{ for } 0<t<T,\\
y(x,0) = y_0(x), z(x,0) = z_0(x), \mbox{ in }  \Omega,
\end{cases}
\end{equation*}
with $y(T)=0$, $z(T)=0$ and $\|\nu\| \leq C(\|y_0\|_{L^2(\Omega)}+\|z_0\|_{L^2(\Omega)})$. 

In the present case, we complete the demonstration using the Theorem of Aubin-Lions. We have $W \subset L^2(Q)$ compactly, and there is a subsequence $(y_n,z_n)$, such that
$$
(y_n,z_n) \to (y,z), \quad \text{strongly in } L^2(Q)\times L^2(Q) = \X,
$$
i.e. $(y,z) \in \Phi(\bar y,\bar z)$ and $\mathbf{\Phi(\bar y,\bar z)}$ \textbf{is compact.}

By the Theorem \ref{theorem_linear}, we have that $\mathbf{\Phi(\bar y,\bar z)}$ \textbf{is nonempty.}

Now we prove that $\Phi(B) \subset B$. If $R>0$ is fixed, let $(\bar y,\bar z) \in B$ and $(y,z) \in \Phi(B)$.
If we choose $\delta>0$ sufficiently small, depending on $T$, such that $\|y_0\|_{L^2(\Omega)} + \|z_0\|_{L^2(\Omega)} < \delta$, then \textbf{we get that} $\mathbf{\Phi(B) \subset B}$, by the energy estimates \eqref{eq:energy_est_L2}.

Now we prove that that graph of $\Phi$ is closed.
Let us consider the sequence $((\bar y_n,\bar z_n),(y_n,z_n)) \in Graph(\Phi)$, such that $((\bar y_n,\bar z_n),(y_n,z_n)) \to ((\bar y,\bar z),( y, z))$ in $\X \times \X$. That is, the functions satisfy the system below:
\begin{equation}\label{sistema_aproximado}
\begin{cases}
y_{n,t}=a_1  y_{n,xx}+b_1 y_{n,x} + c_1 y_n +\lambda_1\int_0^1 J_1(\zeta-x)y_n(\zeta,t) d\zeta +  a(\bar y_n,\bar z_n)y_n +  b(\bar y_n,\bar z_n)z_n + \nu_n 1_\omega, \ \ &\mbox{ in } Q,\\
z_{n,t}=a_2  z_{n,xx}+b_2 z_{n,x} +c_2z_n +\lambda_2\int_0^1J_2(\zeta-x)z_n(\zeta,t)d\zeta+  c(\bar y_n,\bar z_n)y_n +  d(\bar y_n,\bar z_n)z_n, \ \ &\mbox{ in } Q,\\
y_n(0,t) = y_n(1,t)=z_n(0,t) = z_n(1,t)  = 0 ,\ \ 
\mbox{ for } 0<t<T,\\
y_n(x,0) = y_0(x), z_n(x,0) = z_0(x), \mbox{ in }  \Omega,
\end{cases}
\end{equation}
and we have that 
$\nu_n\in L^2(\omega\times(0,T)) \text{, } (y_n,z_n) \text{ is solution of the system with }  \ y_n(T)=0, \ z_n(T)=0  \text{ and } \|\nu_n\|_{L^2(\omega\times(0,T))} \leq C (\|y_0\|_{L^2(\Omega)} + \|z_0\|_{L^2(\Omega)})$.

As above we have the energy estimate  \eqref{eq:energy_est_L2}, because the initial conditions of the system are fixed for every sequence of solutions and the controls are uniformly limited by the norm of these initial conditions. Thus, there is a subsequence also denoted $(y_n,z_n)$ such that
$$
y_n \rightharpoonup y, \qquad z_n \rightharpoonup z, \quad \text{weakly in } L^2(0,T,H^1_0(\Omega)),  
$$
$$
y_{n,t} \rightharpoonup  y_t, \qquad z_{n,t} \rightharpoonup z_t, \quad \text{weakly in } L^2(0,T,H^{-1}(\Omega)).  
$$
%\textcolor{red}{
%Because of an Aubin Lions type argument,
%$$
%y_n \rightharpoonup   y, \qquad z_n \rightharpoonup   z, \quad \text{weakly in } L^2(0,T,H^1_0(\Omega)),  
%$$
%$$
%y_{n,t} \rightharpoonup  y_t, \qquad z_{n,t} \rightharpoonup  z_t, \quad \text{weakly in } L^2(0,T,H^{-1}(\Omega)).  
%$$
%[normalement avec Aubin Lions on obtient convergence forte, conferir se está certo. Está um pouco diferente das notas de aula]
%}

We also have for a subsequence, due to $((\bar y_n,\bar z_n),(y_n,z_n)) \to ((\bar y,\bar z),( y,  z))$ in $\X \times \X$,
$$
\bar y_n \to \bar y, \qquad y_n \to   y, \qquad \bar z_n \to \bar z, \qquad z_n
\to   z.
$$
almost everywhere in $Q$. So, considering that $a$ is a continuous function, we have:

$$
a(\bar y_n,\bar z_n)y_n \to a(\bar y,\bar z) y, \quad \text{almost everywhere} \quad \Omega\times (0,T).
$$

On the other hand, we have that $\int_Q |a(\bar y_n,\bar z_n)y_n|^2 \leq M^2 \int_Q  |y_n|^2 \leq C$. Thus, by the Lemma of Lions, we have 
$$
a(\bar y_n,\bar z_n)y_n \to a(\bar y,\bar z)y \quad \text{in} \quad L^2(Q).
$$

We estimate analogously the other terms. 
$$
b(\bar y_n,\bar z_n)y_n \to b(\bar y,\bar z)y \quad \text{in} \quad L^2(Q).
$$
$$
c(\bar y_n,\bar z_n)y_n \to c(\bar y,\bar z)y \quad \text{in} \quad L^2(Q).
$$
$$
d(\bar y_n,\bar z_n)y_n \to d(\bar y,\bar z)y \quad \text{in} \quad L^2(Q).
$$

Moreover, observing that $\|\nu_n\|_{L^2(\omega\times(0,T))} \leq C (\|y_0\|_{L^2(\Omega)} + \|z_0\|_{L^2(\Omega)})$, we also have
$\nu_n \rightharpoonup \nu$ weakly in $L^2(\omega \times (0,T))$.

Then, we can pass the limit in the system \eqref{sistema_aproximado}. We have
$$y_n(x,0)=y_0(x), \quad z_n(x,0)=z_0(x), \quad y_n(x,T)=0, \quad z_n(x,T)=0.$$

With the same argument as at the end of the proof of Theorem \ref{theorem_linear} we have
$$y(x,0)=y_0(x), \quad z(x,0)=z_0(x), \quad y(x,T)=0, \quad z(x,T)=0.$$

Thus, we have
\begin{equation*}
%\label{eq:sist_aprox_n}
\begin{cases}
y_{t}=a_1  y_{xx}+b_1 y_{x} + c_i y +\lambda_1\int_0^1 J_1(\zeta-x)y(\zeta,t) d\zeta + a(\bar y,\bar z)y + b(\bar y,\bar z)z + \nu 1_\omega, \ \ &\mbox{ in } Q,\\
z_{t}=a_2  z_{xx}+b_2 z_{x} +c_2 z +\lambda_2\int_0^1J_2(\zeta-x)z(\zeta,t)d\zeta+ c(\bar y,\bar z)y + d(\bar y,\bar z)z, \ \ &\mbox{ in } Q,\\
y(0,t) = y(1,t)=z(0,t) = z(1,t)  = 0 ,\ \ 
\mbox{ for } 0<t<T,\\
y(x,0) = y_0(x), z(x,0) = z_0(x), \mbox{ in }  0 < x < 1,
\end{cases}
\end{equation*}
and $\|\nu\|_{L^2(\omega\times (0,T))} \leq \liminf \|\nu_n\| \leq C(\|y_0\|_{L^2(Q)} + \|z_0\|_{L^2(Q)})$.

This shows that $(y,z) \in \Phi(\bar y,\bar z)$, i.e., $\mathbf{\Phi}$ \textbf{has closed graph}.

Thus, by Kakutani's theorem, there exists a fixed point $(y,z) \in \Phi(y,z)$ that satisfies the system \eqref{system_fixed_point} and
$y(T)=0$, $z(T)=0$, with control $\nu \in L^2(\omega \times (0,T))$ such that
$$
\|\nu\|_{L^2(\omega \times (0,T)} \leq C_1(\|y_0\|_{L^2(\Omega)} + \|z_0\|_{L^2(\Omega}).
$$

\end{proof}

\section{Related problems and additional comments}
\label{sec:add_coments}
\begin{itemize}

\item As pointed in Remark \ref{2_controles}, the intermediate Carleman solves the problem with two controls.

\item The adaptation of the case of boundary controls is standard. For instance, for the system \eqref{sistema0} we look for boundary controls $h_1$ and $h_2$ such that the system
\begin{equation*} %\label{sistema_boundary}
\begin{cases}
u_t=a_1  u_{xx}+b_1 u_x + c_1 u +\lambda_1\int_0^1 J_1(z-x)u(z,t) dz + F(u,v), \ \ &\mbox{ in } Q,\\
v_t=a_2  v_{xx}+b_2v_x +c_2v +\lambda_2\int_0^1J_2(z-x)v(z,t)dz+G(u,v), \ \ &\mbox{ in } Q,\\
u(0,t) = v(0,t) = 0, \ \ \mbox{ for } 0<t<T, \\
u(1,t) = h_1(t), \quad v(1,t)  = h_2(t) ,\ \ \mbox{ for } 0<t<T,\\
u(x,0) = u_0(x), v(x,0) = v_0(x), \mbox{ in }  \Omega,
\end{cases}
\end{equation*}
is locally null controllable, where the other parameters are the same as in the original system \eqref{sistema0}. We solve the problem with interior control $\bar \nu$ in the extended interval $\Omega_\varepsilon = (0,1+\varepsilon)$ where the control is acting in an open set $\bar \omega \subset (1,1+\varepsilon)$. Then, the boundary controls are given by $h_1(t)=\bar u(1,t)$ and $h_2(t)=\bar v(1,t)$, where $(\bar u, \bar v)$ denote the solutions given by Theorem \ref{maintheorem0} in the extended interval.
Let us recall that the case with only one boundary control is open, even for a system of coupled heat equations.

    \item The authors are working on the following generalization: for the domains $\Omega = (0,1)$ and $Q= \Omega \times (0,T)$ we consider the null controllability of the following system of semilinearly coupled equations with kernel terms,
\begin{equation*} %\label{sistema1an}
\begin{cases}
y_t=a_1  y_{xx}+b_1 y_x + c_1 y +\lambda_1(t)\int_0^1 J_1(\zeta-x)y(\zeta,t) d\zeta + P(y,z,y_x) + \nu 1_\omega, \ \ &\mbox{ in } Q,\\
z_t=a_2  z_{xx}+b_2 z_x + c_2 z +\lambda_2(t)\int_0^1 J_2(\zeta-x)z(\zeta,t)d\zeta+ Q(y,z,z_x), \ \ &\mbox{ in } Q,\\
y(0,t) = y(1,t) = z(0,t) = z(1,t)  = 0 ,\ \ 
\mbox{ for } 0<t<T,\\
y(x,0) = y_0(x), z(x,0) = z_0(x), \mbox{ in }  \Omega,
\end{cases}
\end{equation*}
where $\nu$ is a control acting on an open set $\omega \subset \Omega$.

    \item Let us consider $\Omega = (0,1)$ and $Q:= \Omega \times (0,T)$. We consider the null controllability of the following system of semilinearly coupled equations with kernel terms,
\begin{equation*} %\label{sistema1an}
\begin{cases}
y_t=a_1  y_{xx}+b_1 y_x + c_1 y +\lambda_1(t)\int_0^1 J_1(\zeta-x)y(\zeta,t) d\zeta + P(y,z,y_x,z_x) + \nu 1_\omega, \ \ &\mbox{ in } Q,\\
z_t=a_2  z_{xx}+b_2 z_x + c_2 z +\lambda_2(t)\int_0^1 J_2(\zeta-x)z(\zeta,t)d\zeta+ Q(y,z,y_x,z_x), \ \ &\mbox{ in } Q,\\
y(0,t) = y(1,t) = z(0,t) = z(1,t)  = 0 ,\ \ 
\mbox{ for } 0<t<T,\\
y(x,0) = y_0(x), z(x,0) = z_0(x), \mbox{ in }  0 < x < 1,
\end{cases}
\end{equation*}
where $\nu$ is a control acting on an open set $\omega \subset (0,1)$. This problem is more difficult because there are only few results and Carleman estimates in the generic case where the linearized system contains terms with both $y$ and $y_x$ or $z$ and $z_x$ and it is not possible to isolate $y$ or $z$ from one of the equations. In that case we mention \cite{Ben-Cris-Gai-Ter-14} where for a linear system, under stronger hypothesis on the coefficients and the geometry of the control set, a null controllability result was proved.

\item From both a mathematical and an applied perspective, an interesting open problem is to establish controllability for systems of $n$ equations using only $n-m$ controls, with $n,m\in\mathbb{N}$ and $n>m>0$. 
Such coupled systems also arise in financial mathematics.
\end{itemize}

%\newpage
%\clearpage
\thispagestyle{plain}        % sem cabeçalho repetido

\end{document}